\documentclass[12pt,reqno] {amsart}%
\usepackage{amsfonts}
\usepackage{amsmath}
\usepackage{amssymb}
\usepackage[all,pdf]{xy}
\usepackage{graphicx}%
\usepackage{hyperref}
\hypersetup{
    colorlinks = true,
    linkcolor = blue,
    urlcolor = blue,
    citecolor = blue,
    anchorcolor = black,
}

\usepackage{cleveref}
\usepackage{comment}
 \usepackage{mathrsfs}
 \usepackage{float}
 \usepackage{tikz}
 \usepackage{bm}
 \usepackage{ulem}
\usepackage{slashed}
 \usepackage{braket}
\usepackage{enumerate}
\usepackage{rotating}
\usepackage{caption}
\usepackage{bbm}

\usepackage{CJKutf8}

\usepackage[a4paper,top=4cm,bottom=4cm,left=3.5cm,right=3.5cm]{geometry}

\allowdisplaybreaks

\makeatletter
\def\l@section{\@tocline{1}{0pt}{1pc}{}{}}
\def\l@subsection{\@tocline{2}{0pt}{1pc}{4.6em}{}}
\def\l@subsubsection{\@tocline{3}{0pt}{1pc}{7.6em}{}}
\renewcommand{\tocsection}[3]{%
  \indentlabel{\@ifnotempty{#2}{\makebox[2.3em][l]{%
    \ignorespaces#1 #2.\hfill}}}#3}
\renewcommand{\tocsubsection}[3]{%
  \indentlabel{\@ifnotempty{#2}{\hspace*{2.3em}\makebox[2.3em][l]{%
    \ignorespaces#1 #2.\hfill}}}#3}
\renewcommand{\tocsubsubsection}[3]{%
  \indentlabel{\@ifnotempty{#2}{\hspace*{4.6em}\makebox[3em][l]{%
    \ignorespaces#1 #2.\hfill}}}#3}
\makeatother
\theoremstyle{plain}

\numberwithin{equation}{section}

\newtheorem{theorem}{Theorem}[section]
\newtheorem{corollary}[theorem]{Corollary}
\newtheorem{lemma}[theorem]{Lemma}
\newtheorem{proposition}[theorem]{Proposition}

\newtheorem{question}[theorem]{Question}
\newtheorem{conjecture}[theorem]{Conjecture}
\newtheorem{main theorem}{Main Theorem}

\theoremstyle{definition}
\newtheorem{definition}[theorem]{Definition}
\newtheorem{remark}[theorem]{Remark}
\newtheorem{example}[theorem]{Example}

\newcommand{\be}{\begin{equation}}
\newcommand{\ee}{\end{equation}}

\renewcommand{\leq}{\leqslant}
\renewcommand{\geq}{\geqslant}

\begin{document}

\title{Free semigroupoid algebras and the first cohomology groups}
\author{Linzhe Huang}
\address{Linzhe Huang, Beijing Institute of Mathematical Sciences and Applications, Beijing, 101408, China}
\email{huanglinzhe@bimsa.cn}
\author{Minghui Ma}
\address{Minghui Ma, School of Mathematical Sciences, Dalian University of Technology, Dalian, 116024, China}
\email{minghuima@dlut.edu.cn}
\date{}

\maketitle

\begin{abstract}
This paper investigates derivations of the free semigroupoid algebra $\mathfrak{L}_G$ of a countable or uncountable directed graph $G$ and its norm-closed version, the tensor algebra $\mathcal{A}_G$.
We first prove a weak Dixmier approximation theorem for $\mathfrak{L}_G$ when $G$ is strongly connected.
Using the theorem, we show that if every connected component of $G$ is strongly connected, then every bounded derivation $\delta$ from $\mathcal{A}_G$ into $\mathfrak{L}_G$ is of the form $\delta=\delta_T$ for some $T\in\mathfrak{L}_G$ with $\|T\|\leqslant\|\delta\|$.
For any finite directed graph $G$, we also show that the first cohomology group $H^1(\mathcal{A}_G,\mathfrak{L}_G)$ vanishes if and only if every connected component of $G$ is either strongly connected or a fruit tree.

To handle infinite directed graphs, we introduce the alternating number and propose \Cref{conj intro-in-tree}.
Suppose every connected component of $G$ is not strongly connected.
We show that if every bounded derivation from $\mathcal{A}_G$ into $\mathfrak{L}_G$ is inner, then every connected component of $G$ is a generalized fruit tree and the alternating number $A(G)$ of $G$ is finite.
The converse is also true if the conjecture holds.

Finally, we provide some examples of free semigroupoid algebras together with their nontrivial first cohomology groups.
\end{abstract}



%\tableofcontents
%\newpage

\section{Introduction}\label{sec introduction}
A {\sl free semigroupoid algebra} is a weak-operator closed algebra generated by partial isometries and projections, associated with a directed graph.
It was introduced by Kribs and Power \cite{KP04} as a generalization of free semigroup algebras.
The theory of free semigroup algebras originates from the research of Popescu \cite{Pop89,Pop89b,Pop91,Pop95} and Arias and Popescu \cite{AP95} where they introduced the non-commutative analytic Toeplitz algebra and proved its reflexivity.
In \cite{DP99}, Davidson and Pitts studied many properties of free semigroup algebras as the axiomatization of the non-commutative analytic Toeplitz algebra, and examined their invariant subspaces.
In a series of papers, free semigroup algebras have been systematically investigated:
the hyperreflexivity, the classification of atomic free semigroup algebras, the structure theorem, the Kaplansky density theorem, and related topics are established, see \cite{Dav06,DKP01,DLP05,DP99,Ken11} and references therein.
In recent years, many of these results were extended to free semigroupoid algebras \cite{DDL19,JP06,JK04,JK05,KK04,KP04}.

The first cohomology groups of operator algebras characterize the structure of their derivations.
The study of derivations on operator algebras dates back to a 1953 conference where Kaplansky asked Singer about derivations on $C(X)$, the algebra of continuous functions on a compact Hausdorff space $X$.
A day later, Singer provided a concise and elegant proof that the only such derivation is the zero mapping.
After that Kaplansky went on to wrote his paper \cite{Kap53} and showed that every derivation on a type $\mathrm{I}$ von Neumann algebra is inner.
%In \cite{Kad66}, Kadison proved that every derivation on a $C^*$-algebra is spatial, more precisely, it can be implement by an operator in the weak-operator closure of the $C^*$-algebra.
In \cite{Kad66,Sak66}, it was proved by Kadison and Sakai that every derivation on a von Neumann algebra is inner, i.e., the first cohomology group of a von Neumann algebra is trivial.
Later in \cite{Chr77}, Christensen proved that every derivation on a nest algebra is inner, which is the case of non-self-adjoint algebras.

%Results of derivations on non-self-adjoint algebras are much less than self-adjoint algebras.
Free semigroupoid algebras constitute an important class of non-self-adjoint algebras.
This paper aims to provide a complete characterization of derivations on these algebras for arbitrary directed graphs.
In the special case where the directed graph consists of a single vertex with $n$ distinct loops, the free semigroupoid algebras reduce to the non-commutative analytic Toeplitz algebras.
Previous work by the authors and Wei in \cite{HMW24} established that all bounded derivations on such algebras are inner.
We also refer the reader to \cite{Huang25} for further related results, which proved the innerness of derivations on Thompson's semigroup algebra.

For a directed graph $G$ with vertices $\mathcal{V}=\mathcal{V}(G)$ and edges $\mathcal{E}=\mathcal{E}(G)$, each edge $e$ in $\mathcal{E}$ has a source vertex $s(e)$ and a range vertex $r(e)$ in $\mathcal{V}$.
Let $\mathbb{F}_{G}^+$ be the set of all directed paths of finite length in $G$, which is called the {\sl free semigroupoid} of $G$.
Let $\mathcal{H}_G=\ell^2(\mathbb{F}_{G}^+)$ be the Hilbert space with an orthonormal basis $\{\xi_w\colon w\in\mathbb{F}_{G}^+\}$.
For each edge $e\in\mathcal{E}$ and each vertex $x\in\mathcal{V}$, the partial isometry $L_e$ and the projection $L_x$ on $\mathcal{H}_G$ are defined in \eqref{equ left}.
The {\sl left free semigroupoid algebra} $\mathfrak{L}_G$ determined by $G$ is the weak-operator closed algebra generated by $\{L_e,L_x\colon e\in\mathcal{E},x\in\mathcal{V}\}$.
The {\sl tensor algebra} $\mathcal{A}_G$ is its norm-closed version.

A directed graph $G$ is called {\sl strongly connected} if there is a directed path from $x$ to $y$ for any $x,y\in\mathcal{V}$.
Suppose $G$ is a finite strongly connected directed graph.
We first show that every bounded derivation $\delta$ from $\mathcal{A}_G$ into $\mathfrak{L}_G$ is inner, i.e., there exists $T\in\mathfrak{L}_G$ such that $\delta=\delta_T$ (see \Cref{prop S-C-finite}).
In order to generalize this result to infinite directed graphs, it is crucial to estimate the norm of $T$, which should be independent of the number of vertices.
For the purpose, we prove a weak Dixmier approximation theorem for $\mathfrak{L}_G$ when $G$ is strongly connected (see \Cref{thm weak-Dixmier}), which states that for any operator $A\in\mathfrak{L}_G$ the intersection of the weak-operator closure of $\{\frac{1}{n}\sum_{j=1}^nU_j^*AU_j\colon U_j~\text{is an isometry in}~\mathfrak{L}_G\}$ and the center $Z(\mathfrak{L}_G)$ of $\mathfrak{L}_G$ is nonempty.
In particular, we are able to prove the following result.

\begin{theorem}[\Cref{thm S-C} and \Cref{cor S-C}]\label{thm intro-S-C}
Suppose $G$ is a directed graph such that every connected component of $G$ is strongly connected.
Then for every bounded derivation $\delta$ from $\mathcal{A}_G$ into $\mathfrak{L}_G$, there exists $T\in\mathfrak{L}_G$ such that $\delta=\delta_T$ and $\|T\|\leqslant\|\delta\|$.
\end{theorem}

We can immediately obtain from \Cref{thm intro-S-C} that every derivation $\delta$ on $\mathcal{A}_G$ can be implemented by an operator $T$ in $\mathfrak{L}_G$ (see \Cref{cor S-C-2}).
This can be regarded as an analogue of Kadison's result \cite{Kad66} that every derivation on a $C^*$-algebra is implemented by an operator in its weak-operator closure.
In \cite{Sak68}, Sakai showed that every derivation on a simple $C^*$-algebra is inner.
Thus, it is natural to propose the following question in the case of tensor algebras.

\begin{question}\label{ques inner}
Let $G$ be a directed graph.
When is each derivation on $\mathcal{A}_G$ inner?
\end{question}

It is related to a paper by Katsoulis \cite{Kat16} where he pointed out that if every derivation on $\mathcal{A}_G$ is inner then local derivations are derivations.
For connected directed graphs which are not strongly connected, we obtain a complete answer to this question in \Cref{cor converse}.
For strongly connected directed graphs, we get a partial answer to \Cref{ques inner} in \Cref{lem infinite-not-inner}, \Cref{prop S-C-uncountable}, and \Cref{prop circle-inner}.
We provide a sufficient condition for this question when $\mathcal{A}_G$ is the non-commutative disc algebra in \Cref{lemma condition-inner}.
In \Cref{prop circle-inner-2}, we also prove that every derivation on $\mathfrak{L}_G$ is inner if every connected component of $G$ is an $n$-circle graph, which provides an affirmative answer to the question proposed at the end of \cite{Dun07}.

One of the main purpose of this paper is to characterize all the directed graphs $G$ such that all bounded derivations from $\mathcal{A}_G$ into $\mathfrak{L}_G$ are inner, or equivalently, $H^1(\mathcal{A}_G,\mathfrak{L}_G)=0$.
By \Cref{thm intro-S-C}, it remains to investigate the situation when every connected component of $G$ is not strongly connected.
The following is the fundamental case.
A tree $G$ is called a {\sl rooted tree} if there exists a vertex $v_0$ satisfying one of the following conditions:
(i) there is a path from $v_0$ to every vertex $v$;
(ii) there is a path from every vertex $v$ to $v_0$.
In case (i) $G$ is called an {\sl out-tree} and in case (ii) $G$ is called an {\sl in-tree}.

We prove that for every derivation $\delta$ on $\mathcal{A}_G$, there exists $T\in\mathcal{A}_G$ such that $\delta=\delta_T$ and $\|T\|\leqslant\|\delta\|$ if $G$ is an out-tree (see \Cref{prop out-tree}).
For in-trees, it is not difficult to show that $T$ can be chosen such that $\|T\|\leqslant a_n\|\delta\|$, where $a_n$ is a universal constant that is dependent on the number of the vertices $n$.
The essential problem is whether the constant can be chosen to be independent of the number of the vertices.
The difference between out-trees and in-trees is that every out-tree can be isometrically embedded into the upper triangle algebra (see \Cref{lem out-tree}) while in-trees may not be.
We compute an example in \Cref{prop in-tree-eg} and propose the following conjecture.

\begin{conjecture}[\Cref{conj in-tree}]\label{conj intro-in-tree}
Suppose $G$ is an in-tree.
Then for every derivation $\delta$ on $\mathcal{A}_G$, there exists $T\in\mathcal{A}_G$ such that $\delta=\delta_T$ and $\|T\|\leqslant C\|\delta\|$, where $C\geqslant 1$ is a universal constant.
\end{conjecture}

A generalized tree is a connected directed graph with at least two vertices and no polygon.
A tree is a generalized tree with finitely many vertices.
For convenience, we assume that every tree contains at least two vertices (see \Cref{def tree}).

Suppose $G_t$ is a generalized tree, and $V:=\{v_\lambda\}_{\lambda\in\Lambda}$ is a subset of leaves of $G_t$.
Suppose $G_f$ is a directed graph with connected components $\{\mathscr{C}_{n_\lambda}\}_{\lambda\in\Lambda}$, and each $v_\lambda$ is identified with a vertex of the $n_\lambda$-circle graph $\mathscr{C}_{n_\lambda}$.
The amalgamated graph $G=G_t\sqcup_VG_f$ is called a {\sl generalized fruit tree} provided $V\ne\mathcal{V}_t$ (see \Cref{def fruit-tree}).
The alternating number of a directed graph is given in \Cref{def alternating}.

\begin{theorem}[\Cref{cor inner-imply} and \Cref{cor fruit-tree-bound}]
\label{thm intro-fruit-tree}
Suppose $G$ is a directed graph with connected components $\{G_\lambda\}_{\lambda\in\Lambda}$ such that each $G_\lambda$ is not strongly connected.
Let $A(G_\lambda)$ be the alternating number of $G_\lambda$.
Then we have the following two statements:
\begin{enumerate}[(i)]
\item If every bounded derivation from $\mathcal{A}_G$ into $\mathfrak{L}_G$ is inner, then each $G_\lambda$ is a generalized fruit tree and $A(G)=\sup_\lambda A(G_\lambda)<\infty$.

\item Assume that \Cref{conj intro-in-tree} holds. If each $G_\lambda$ is a generalized fruit tree and $A(G)<\infty$, then every bounded derivation from $\mathcal{A}_G$ into $\mathfrak{L}_G$ is inner.
\end{enumerate}
\end{theorem}
As a consequence of \Cref{thm intro-S-C} and \Cref{thm intro-fruit-tree}, we can obtain that if $G$ is a finite directed graph, then $H^1(\mathcal{A}_G,\mathfrak{L}_G)=0$ if and only if every connected component of $G$ is either strongly connected or a fruit tree.
Moreover, if $G$ is connected but not strongly connected, then $H^1(\mathcal{A}_G,\mathcal{A}_G)=0$ if and only if $G$ is a fruit tree that contains no in-fruit (see \Cref{cor converse}).

This paper is structured as follows.
In \Cref{sec pre}, we review and prove some basic properties of free semigroupoid algebras that will be used throughout the paper.
In \Cref{sec derivation}, we present several auxiliary results on derivations that will simplify the arguments in later sections.
We establish a weak Dixmier approximation theorem for free semigroupoid algebras in \Cref{sec weak-Dixmier}, which plays a key role in the characterization of derivations in \Cref{sec S-C}.
\Cref{sec S-C} is devoted to the study of derivations in the case where every connected component of the directed graph $G$ is strongly connected and provides the proof of \Cref{thm intro-S-C}.
In \Cref{sec fruit-tree}, we address the case where every connected component of $G$ is not strongly connected and prove \Cref{thm intro-fruit-tree}.
Finally, in \Cref{sec cohomology}, we compute the first cohomology groups of some examples.

\section{Preliminaries}\label{sec pre}
\subsection{Free semigroupoids}
Let $G=(\mathcal{V},\mathcal{E},s,r)$ be a {\sl directed graph} with countably or uncountably many vertices $\mathcal{V}=\mathcal{V}(G)$ and edges $\mathcal{E}=\mathcal{E}(G)$.
Each edge $e\in\mathcal{E}$ has an initial vertex $x$ and a final vertex $y$, and we view $e$ as an arrow from $x$ to $y$.
We shall use the {\sl source map} $s$ and the {\sl range map} $r$ to denote the vertices, i.e., $s(e)=x$ and $r(e)=y$.
See the following graph
\begin{equation*}
\xymatrix{
r(e)=y&s(e)=x.
\ar_{e}"1,2";"1,1"
}
\end{equation*}
We denote by $\mathbb{F}_{G}^+$ the set of all directed paths of finite length in $G$, which is called the {\sl free semigroupoid} or the {\sl path space} of $G$.
Vertices are considered as paths of length zero in $\mathbb{F}_{G}^+$.
The source and range maps can be naturally extended to $\mathbb{F}_{G}^+$, which are still denoted by $s$ and $r$.
More precisely, for a path of length zero, i.e., a vertex $v\in\mathcal{V}$, we set $r(v)=s(v)=v$.
For a path $p=e_1e_2\cdots e_n\in\mathbb{F}_{G}^+$ of length $\ell(p)=n\geqslant 1$, where $e_1,e_2,\ldots,e_n\in\mathcal{E}$ and $s(e_j)=r(e_{j+1})$ for each $1\leqslant j\leqslant n-1$, we define $r(p)=r(e_1)$ and $s(p)=s(e_n)$.
See the following graph
\begin{equation*}
\xymatrix{
r(p)&v_2&v_3&\cdots&s(p).
\ar_{e_1}"1,2";"1,1"
\ar_{e_2}"1,3";"1,2"
\ar_{e_3}"1,4";"1,3"
\ar_{e_n}"1,5";"1,4"
}
\end{equation*}
If $p_1,p_2\in\mathbb{F}_{G}^+$ with $s(p_1)=r(p_2)$, then their product $p_1p_2\in\mathbb{F}_{G}^+$ is well-defined.
This operation makes $\mathbb{F}_{G}^+$ into a semigroupoid.
In particular, if $\mathcal{V}$ is a singleton, then $\mathbb{F}_{G}^+$ is the free semigroup generated by $\mathcal{E}$.

A path $p\in\mathbb{F}_{G}^+$ of length $\ell(p)\geqslant 1$ is called a {\sl circle} at a vertex $v\in\mathcal{V}$ if $r(p)=s(p)=v$.
If $p$ is a circle at $v$, then $p^j$ is meaningful for every $j\geqslant 1$.
For convenience, we set $p^0:=v$.
A length-one circle is said to be a {\sl loop}.
A circle is called {\sl minimal} if it is not the products of any two circles.
For every vertex $v\in\mathcal{V}$, let $\mathcal{E}_v$ denote the set of all minimal circles at $v$, and $G_v$ the directed graph with $\mathcal{V}(G_v)=\{v\}$ and $\mathcal{E}(G_v)=\mathcal{E}_v$.
Then $\mathbb{F}_{G_v}^+$ is the free semigroup generated by $\mathcal{E}_v$.
It follows from \cite[Proposition 5.1]{DDL19} that
\begin{equation}\label{equ reduced-by-v}
  \{p\in\mathbb{F}_{G}^+\colon r(p)=s(p)=v\}=\mathbb{F}_{G_v}^+.
\end{equation}
The following lemma is a direct application of \cite[Lemma 2.1]{HMW24}.

\begin{lemma}\label{lem 2.1}
Let $w,p_1,p_2\in\mathbb{F}_{G}^+$ be circles at a vertex $v\in\mathcal{V}$.
If there exists an integer $k\geqslant\ell(p_2)+1$ such that $p_1w^k=w^kp_2$, then $p_1w=wp_1$ and $p_1=p_2$.
\end{lemma}

A directed graph $G$ is called {\sl strongly connected} or {\sl transitive} if for every pair of vertices $(x,y)$, there exist paths from $x$ to $y$ and from $y$ to $x$.
For every $n\geqslant 1$, the {\sl $n$-circle graph} $\mathscr{C}_n$ is a directed graph with $n$ distinct vertices $\{v_1,v_2,\ldots,v_n\}$ and edges $\{e_1,e_2,\ldots,e_n\}$ satisfying that $r(e_j)=v_j$ and $s(e_j)=v_{j+1\,(\mathrm{mod}\,n)}$ for all $1\leqslant j\leqslant n$.
The following is the $3$-circle graph
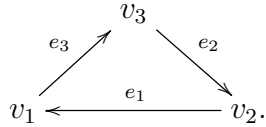
\begin{equation*}
\xymatrix{
&v_3&\\
v_1&&v_2.
\ar^{e_3}"2,1";"1,2"
\ar_{e_1}"2,3";"2,1"
\ar^{e_2}"1,2";"2,3"
}
\end{equation*}
We say that $G$ is a {\sl vertex graph} or {\sl trivial graph} if it consists of one vertex and no edge.
Note that a strongly connected directed graph $G$ is not an $n$-circle graph if and only if one of the following holds:
\begin{enumerate}[(i)]
\item $G$ is a vertex graph;
\item there exist at least two distinct minimal circles at each vertex.
\end{enumerate}

\subsection{Free semigroupoid algebras}
Let $\mathcal{H}_G$ be the Hilbert space $\ell^2(\mathbb{F}_{G}^+)$ with an orthonormal basis $\{\xi_w\colon w\in\mathbb{F}_{G}^+\}$.
For each edge $e\in\mathcal{E}$ and vertex $x\in\mathcal{V}$, we can define partial isometries and projections on $\mathcal{H}_G$ by
\begin{align}\label{equ left}
    L_e\xi_w=\begin{cases}
        \xi_{ew}& r(w)=s(e)\\
        0&\text{otherwise}
    \end{cases}\quad\text{and}\quad
    L_x\xi_w=\begin{cases}
        \xi_{w}& r(w)=x\\
        0&\text{otherwise}
    \end{cases}
\end{align}
and
\begin{align}\label{equ right}
    R_e\xi_w=\begin{cases}
        \xi_{we}& s(w)=r(e)\\
        0&\text{otherwise}
    \end{cases}\quad\text{and}\quad
    R_x\xi_w=\begin{cases}
        \xi_{w}& s(w)=x\\
        0&\text{otherwise}
    \end{cases}.
\end{align}
The {\sl left} and {\sl right free semigroupoid algebras} determined by $G$ are the weak-operator closed algebras given by
\begin{align*}
    \mathfrak{L}_G:=\overline{\mathrm{Alg}}^{\mathrm{WOT}}
    \{L_e,L_x\colon e\in\mathcal{E},x\in\mathcal{V}\},
\end{align*}
and
\begin{align*}
    \mathfrak{R}_G:=\overline{\mathrm{Alg}}^{\mathrm{WOT}}
    \{R_e,R_x\colon e\in\mathcal{E},x\in\mathcal{V}\}.
\end{align*}
Clearly, for any $x,y\in\mathcal{V}$, $L_y\mathfrak{L}_GL_x\ne\{0\}$ if and only if there exists a path from $x$ to $y$.
It follows from \cite[Theorem 4.2]{KP04} that
\begin{equation*}
  \mathfrak{L}_G'=\mathfrak{R}_G,\quad \mathfrak{R}_G'=\mathfrak{L}_G.
\end{equation*}
In particular, both $\mathfrak{L}_G$ and $\mathfrak{R}_G$ are unital.
The {\sl tensor algebra} $\mathcal{A}_G$ is the norm-closed algebra given by
\begin{align*}
    \mathcal{A}_G:=\overline{\mathrm{Alg}}^{\|\cdot\|}
    \{L_e,L_x\colon e\in\mathcal{E},x\in\mathcal{V}\},
\end{align*}
which is a Banach subalgebra of $\mathfrak{L}_G$.
A directed graph with finitely many vertices is called a {\sl finite directed graph} (may have infinitely many edges).
In the case of finite directed graphs, $\mathcal{A}_G$ is called the {\sl quiver algebra} by Muhly and Solel \cite{MS99}.
Note that $\mathcal{A}_G$ is unital when $G$ is a finite directed graph, however, $\mathcal{A}_G$ is nonunital when $G$ is an infinite directed graph.
For both finite and infinite directed graphs, we investigate derivations of $\mathcal{A}_G$ in this paper.
For each $A\in\mathfrak{L}_G$, it has a {\sl Fourier expansion} of the form $\sum_{w\in\mathbb{F}_{G}^+}a_wL_w$, which means that
\begin{equation*}
  A\xi_e=\sum_{s(w)=r(e)} a_w\xi_{we},\quad A\xi_{x}=\sum_{s(w)=x}a_w\xi_{w}
\end{equation*}
for all $e\in\mathcal{E}$ and $x\in\mathcal{V}$ (see \cite[Remark 4.3]{KP04}).
It is clear that every operator in $\mathfrak{L}_G$ is completely determined by its Fourier expansion.
The following result will be employed to handle directed graphs with infinitely many vertices.

\begin{proposition}\label{prop compact}
Any norm-closed ball of $\mathfrak{L}_G$ is weak-operator compact.
\end{proposition}

\begin{proof}
We only need to prove that the unit ball $(\mathfrak{L}_G)_1$ of $\mathfrak{L}_G$ is weak-operator compact.
Since the unit ball $(B(\mathcal{H}))_1$ of $B(\mathcal{H})$ is weak-operator compact and
\begin{equation*}
  (\mathfrak{L}_G)_1=\mathfrak{L}_G\cap(B(\mathcal{H}))_1
\end{equation*}
is weak-operator closed, $(\mathfrak{L}_G)_1$ is weak-operator compact.
\end{proof}

\subsection{Reduced subgraphs and subalgebras}
Let $G=(\mathcal{V},\mathcal{E},r,s)$ be a directed graph and let $F$ be a subset of $\mathcal{V}$.
A path $p=e_1e_2\cdots e_n\in\mathbb{F}_{G}^+$ of length $\ell(p)=n\geqslant 1$ is called a {\sl minimal path} at $F$ if
\begin{equation*}
  r(e_1),s(e_n)\in F~\text{and}~
  s(e_j)=r(e_{j+1})\in\mathcal{V}\backslash F~\text{for each}~1\leq j\leq n-1.
\end{equation*}
In particular, for every vertex $v\in\mathcal{V}$, a minimal path at $\{v\}$ is exactly a minimal circle at $v$.
Let $\mathcal{E}_F$ be the set of all minimal paths at $F$.

\begin{definition}
The {\sl reduced subgraph} $G_F$ of $G$ by $F$ is defined to be the directed graph with vertices $\mathcal{V}(G_F)=F$ and edges $\mathcal{E}(G_F)=\mathcal{E}_F$.
\end{definition}

Note that the reduced subgraph of $G$ is different from the {\sl induced subgraph} of $G$ by $F$, whose edge set is $\{e\in\mathcal{E}\colon r(e),s(e)\in F\}$.
Similar to \eqref{equ reduced-by-v}, we have
\begin{equation}\label{equ reduced-by-F}
  \{p\in\mathbb{F}_{G}^+\colon r(p),s(p)\in F\}=\mathbb{F}_{G_F}^+.
\end{equation}
For simplicity, we denote by $L_F$ the projection $\sum_{v\in F}L_v$.
The following proposition establishes the connection between the reduced subgraphs and reduced subalgebras (see also the proof of \cite[Lemma 2.6]{HMW24}).

\begin{proposition}\label{prop reduced}
Suppose $G$ is directed graph.
Let $F$ be a subset of $\mathcal{V}$ and $G_F$ the reduced graph of $G$ by $F$.
Then the operator algebras $L_F\mathfrak{L}_GL_F$ and $L_F\mathcal{A}_GL_F$ acting on $L_F\mathcal{H}_G$ are completely isometrically isomorphic to $\mathfrak{L}_{G_F}$ and $\mathcal{A}_{G_F}$ acting on $\mathcal{H}_{G_F}$, respectively.
\end{proposition}

\begin{proof}
For any $v\in\mathcal{V}$, let $\Lambda_{F,v}$ be the set of all paths $p\in\mathbb{F}_G^+$ such that $r(p)\in F$, $s(p)=v$, and $p$ does not contain any minimal path at $F$.
More precisely, $\Lambda_{F,v}=\{v\}$ for $v\in F$ and
\begin{equation*}
  \Lambda_{F,v}=\{p=e_1e_2\cdots e_n\in\mathbb{F}_{G}^+\colon r(e_1)\in F,s(e_n)=v,r(e_j)\notin F, 2\leqslant j\leqslant n\}
\end{equation*}
for $v\in\mathcal{V}\backslash F$.
Let $\Lambda_F=\bigcup_{v\in\mathcal{V}}\Lambda_{F,v}$ and $\Omega_F=\{w\in\mathbb{F}_{G}^+\colon r(w)\in F\}$.
Combining with \eqref{equ reduced-by-F}, we have
\begin{align*}
    \Omega_F=\bigsqcup_{p\in\Lambda_F}\mathbb{F}_{G_F}^+p,\quad
    L_F\mathcal{H}_G=\bigoplus_{p\in\Lambda_F}R_p\mathcal{H}_{G_F}.
\end{align*}
Thus, every vector $x\in L_F\mathcal{H}_G$ is of the form $\sum_{p\in\Lambda_F}R_px_p$, where $x_p\in\mathcal{H}_{G_F}$.
Note that $x_p=R_p^*x$.
Therefore, for any $A\in L_F\mathfrak{L}_GL_F$, we have
\begin{equation*}
  A=\bigoplus_{p\in\Lambda_F}R_p(A|_{\mathcal{H}_{G_F}})R_p^*.
\end{equation*}
It follows that the map $A\mapsto A|_{\mathcal{H}_{G_F}}$ from $L_F\mathfrak{L}_GL_F$ onto $\mathfrak{L}_{G_F}$ is a complete isometric isomorphism.
\end{proof}
%Let $\mathbb{D}=\{z\in\mathbb{C}:|z|<1\}$ be the open unit  disc.
We refer to e.g. \cite[Chapter 6]{Dou98} for the definitions and properties of the Hardy spaces.
\begin{corollary}\label{cor relative-commutant}
Suppose $G$ is a directed graph.
Let $w$ be a circle at $v\in\mathcal{V}$.
Then there exists a largest integer $m\geqslant 1$ and a circle $w_1$ at $v$ such that $w=w_1^m$.
Moreover, the relative commutant $\{L_w\}'\cap L_v\mathfrak{L}_GL_v$ is the weak-operator closed algebra generated by $\{L_v,L_{w_1}\}$, which is completely isometrically isomorphic to $H^\infty(\mathbb{D})$ acting on $H^2(\mathbb{D})$.
\end{corollary}

\begin{proof}
By \eqref{equ reduced-by-v} and \Cref{prop reduced}, $L_{v_1}\mathfrak{L}_GL_{v_1}$ is completely isometrically isomorphic to the free semigroup algebra $\mathfrak{L}_{\mathbb{F}_{G_v}^+}$.
Therefore, we only need to focus on free semigroup algebras.

Without loss of generality, we assume that $\mathcal{V}=\{v\}$ and $G=G_v$.
In this case, $\mathbb{F}_G^+$ is a free semigroup.
Clearly, there exists a largest integer $m\geqslant 1$ and a circle $w_1\in\mathbb{F}_G^+$ at $v$ such that $w=w_1^m$.
Moreover, every operator $A\in\{L_w\}'\cap\mathfrak{L}_{\mathbb{F}_G^+}$ has a Fourier expansion of the form $\sum_{j=0}^{\infty}a_jL_{w_1^j}$.
Similar to the proof of \Cref{prop reduced}, $\{L_w\}'\cap\mathfrak{L}_{\mathbb{F}_G^+}$ on $\ell^2(\mathbb{F}_G^+)$ is completely isometrically isomorphic to the weak-operator closed algebra generated by $\{L_v,L_{w_1}\}$ on the Hilbert subspace with an orthonormal basis $\{\xi_{w_1^j}\colon j\geqslant 0\}$.
This completes the proof.
\end{proof}

\begin{remark}
For each $k\geqslant 1$, we have
\begin{equation*}
  \{L_{w^k}\}'\cap L_v\mathfrak{L}_GL_v=\{L_w\}'\cap L_v\mathfrak{L}_GL_v
  \cong H^\infty(\mathbb{D}).
\end{equation*}
\end{remark}

\subsection{Conditional expectation}
Let $\mathcal{B}$ be a Banach algebra and $\mathcal{A}$ a Banach subalgebra of $\mathcal{B}$.
A linear map $E\colon\mathcal{B}\to\mathcal{A}$ is called a {\sl conditional expectation} from $\mathcal{B}$ onto $\mathcal{A}$ if the following conditions hold:
\begin{enumerate}[(i)]
\item $E$ is an idempotent, i.e., $E^2=E$;
\item $E$ is a contraction, i.e., $\|E\|\leqslant 1$;
\item $E$ is an $\mathcal{A}$-bimodule map, i.e., $E(A_1BA_2)=A_1E(B)A_2$ for all $A_1,A_2\in\mathcal{A}$ and $B\in\mathcal{B}$.
\end{enumerate}

\begin{proposition}\label{prop E}
Suppose $G$ is a directed graph.
Let $w$ be a circle at $v\in\mathcal{V}$.
Define a map from $\mathfrak{L}_G$ into $\mathfrak{L}_G$ by
\begin{equation*}
  E_w(A)=\text{\scriptsize {\rm WOT-}}\lim_{k\to \infty}L_w^{*k}AL_w^k.
\end{equation*}
Then $E_w$ is a conditional expectation from $\mathfrak{L}_G$ onto $\{L_w\}'\cap L_v\mathfrak{L}_GL_v$.
\end{proposition}

\begin{proof}
Assume that $A\in\mathfrak{L}_G$ has a Fourier expansion $\sum_{p\in\mathbb{F}_G^+}a_pL_p$.
For any elements $p_1,p_2\in\mathbb{F}_G^+$, we have
\begin{equation*}
  \langle L_w^{*k}AL_w^k\xi_{p_1},\xi_{p_2}\rangle
  =\sum_{p\in\mathbb{F}_G^+}a_p\langle\xi_{pw^kp_1},\xi_{w^kp_2}\rangle.
\end{equation*}
By \Cref{lem 2.1}, if the equation $pw^kp_1=w^kp_2$ has a solution $p$ for some large integer $k$, then we must have $pw=wp$ and $pp_1=p_2$.
It follows that
\begin{equation*}
  \lim_{k\to\infty}\langle L_w^{*k}AL_w^k\xi_{p_1},\xi_{p_2}\rangle
  =\sum_{pw=wp}a_p\langle L_p\xi_{p_1},\xi_{p_2}\rangle.
\end{equation*}
Since $\{L_w^{*k}AL_w^k\}_{k=1}^{\infty}$ is a bounded sequence, $E_w(A)$ is well-defined and has a Fourier expansion $\sum_{pw=wp}a_pL_p$.
This completes the proof.
\end{proof}

\begin{remark}
It is clear that $E_{w^k}=E_w$ for all integer $k\geqslant 1$.
\end{remark}

The following corollary will be used in the proof of \Cref{thm weak-Dixmier}.

\begin{corollary}\label{cor sum-U}
Suppose $G$ is a strongly connected directed graph that is not an $n$-circle graph.
Let $\{v_j\}_{j=1}^n$ be $n$ distinct vertices of $G$.
Then there are partial isometries $\{U_{i,k}\colon 1\leqslant i\leqslant n,k\geqslant 1\}$ in $\mathfrak{L}_G$ such that $U_{i,k}^*U_{i,k}=\sum_{j=1}^nL_{v_j}$ and
\begin{align*}
   \text{\scriptsize {\rm WOT-}}\lim_{k\to\infty}\sum_{i=1}^{n} U_{i,k}^*AU_{i,k}
   =\left(\sum_{j=1}^n\lambda_{j}(A)\right)\left(\sum_{j=1}^nL_{v_j}\right)
\end{align*}
for all $A\in\mathfrak{L}_G$, where $\lambda_j(A)$ is the Fourier coefficient of $A$ at $L_{v_j}$.
\end{corollary}

\begin{proof}
If $G$ is a vertex graph, then $\mathfrak{L}_G\cong\mathbb{C}$ and the conclusion is clear.
We assume that $G$ is not a vertex graph.
Let $c_j$ and $d_j$ be two distinct minimal circles at $v_j$ for every $1\leqslant j\leqslant n$, $p_{ij}$ a path from $v_j$ to $v_i$ for every $1\leqslant i,j\leqslant n$, $\{m_j\}_{j=1}^{n}$ a set of positive integers such that
\begin{equation*}
  m_1\ell(c_1)<m_2\ell(c_2)<\cdots<m_n\ell(c_n),
\end{equation*}
and $w_j=c_j^{m_j}$ for every $1\leqslant j\leqslant n$.
For every $1\leqslant i\leqslant n$ and $k\geqslant 1$, we define
\begin{align*}
   U_{i,k}=\sum_{j=1}^nL_{w_j^kd_jp_{j,j+i\,(\mathrm{mod}\,n)}}.
\end{align*}
Let $A\in\mathfrak{L}_G$.
Then for any $1\leqslant i\ne j\leqslant n$, we have
\begin{equation*}
  \text{\scriptsize {\rm WOT-}}\lim_{k\to\infty}
  L_{w_i^k}^*AL_{w_j^k}=0.
\end{equation*}
Note that $L_{d_j}^*L_{c_j}=0$ for each $1\leqslant j\leqslant n$.
Combining with \Cref{prop E}, we have
\begin{align*}
   \text{\scriptsize {\rm WOT-}}\lim_{k\to\infty}U_{i,k}^*AU_{i,k}
   &=\sum_{j=1}^{n}L_{d_jp_{j,j+i\,(\mathrm{mod}\,n)}}^*E_{c_j}(A)
   L_{d_jp_{j,j+i\,(\mathrm{mod}\,n)}}\\
   &=\sum_{j=1}^n\lambda_j(A)L_{v_{j+i\,(\mathrm{mod}\,n)}}.
\end{align*}
It is clear that
\begin{align*}
   \sum_{i=1}^n\sum_{j=1}^n\lambda_j(A)L_{v_{j+i\,(\mathrm{mod}\,n)}}
   =\sum_{j=1}^n\lambda_j(A)\sum_{i=1}^nL_{v_{j+i\,(\mathrm{mod}\,n)}}
   =\sum_{j=1}^n\lambda_j(A)\sum_{i=1}^nL_{v_i}.
\end{align*}
This completes the proof.
\end{proof}

\section{Auxiliary results on derivations}\label{sec derivation}
In this section, we prove some preliminary results on derivations.
Let $\mathcal{A}$ be a Banach algebra and $\mathcal{X}$ a Banach $\mathcal{A}$-bimodule.
A linear map $\delta\colon\mathcal{A}\to\mathcal{X}$ is called a {\sl derivation} if $\delta(AB)=\delta(A)B+A\delta(B)$ for all $A,B\in\mathcal{A}$.
Every element $T\in\mathcal{X}$ induces a derivation $\delta_T$, which is defined by $\delta_T(A)=AT-TA$.
It is clear that $\delta_T$ is bounded with $\|\delta_T\|\leqslant 2\|T\|$.
Such derivations are called {\sl inner derivations}.
Let $\mathrm{Der}(\mathcal{A},\mathcal{X})$ and $\mathrm{Inn}(\mathcal{A},\mathcal{X})$ be the linear spaces of all bounded derivations and inner derivations from $\mathcal{A}$ into $\mathcal{X}$, respectively.
The {\sl first cohomology group} $H^1(\mathcal{A},\mathcal{X})$ is defined to be the quotient space $\mathrm{Der}(\mathcal{A},\mathcal{X})/\mathrm{Inn}(\mathcal{A},\mathcal{X})$.

\subsection{Decomposition of derivations}
Let $\{\mathcal{A}_\lambda\}_{\lambda\in\Lambda}$ be a family of Banach algebras and $\{\mathcal{X}_\lambda\}_{\lambda\in\Lambda}$ a family of Banach spaces such that $\mathcal{X}_\lambda$ is a Banach $\mathcal{A}_\lambda$-bimodule for each $\lambda\in\Lambda$.
We define a Banach algebra by
\begin{equation*}
  \ell^\infty(\{\mathcal{A}_\lambda\}_{\lambda\in\Lambda})
  =\left\{(A_\lambda)_{\lambda\in\Lambda}\colon A_\lambda\in\mathcal{A}_\lambda, \sup_{\lambda\in\Lambda}\|A_\lambda\|<\infty\right\}.
\end{equation*}
Let $c_0(\{\mathcal{A}_\lambda\}_{\lambda\in\Lambda})$ be the norm-closure of the set of all elements with finitely many nonzero terms in $\ell^\infty(\{\mathcal{A}_\lambda\}_{\lambda\in\Lambda})$.
We can similarly define the Banach space $\ell^\infty(\{\mathcal{X}_\lambda\}_{\lambda\in\Lambda})$.
Then $\ell^\infty(\{\mathcal{X}_\lambda\}_{\lambda\in\Lambda})$ is a Banach bimodule over $\ell^\infty(\{\mathcal{A}_\lambda\}_{\lambda\in\Lambda})$ in a natural way, i.e.,
\begin{equation*}
  (A_\lambda)_{\lambda\in\Lambda}(X_\lambda)_{\lambda\in\Lambda}
  (B_\lambda)_{\lambda\in\Lambda}
  =(A_\lambda X_\lambda B_\lambda)_{\lambda\in\Lambda}
\end{equation*}
for all $(A_\lambda)_{\lambda\in\Lambda},(B_\lambda)_{\lambda\in\Lambda}
\in\ell^\infty(\{\mathcal{A}_\lambda\}_{\lambda\in\Lambda})$, and $(X_\lambda)_{\lambda\in\Lambda}
\in\ell^\infty(\{\mathcal{X}_\lambda\}_{\lambda\in\Lambda})$.
Clearly, $\ell^\infty(\{\mathcal{X}_\lambda\}_{\lambda\in\Lambda})$ is also a Banach bimodule over $c_0(\{\mathcal{A}_\lambda\}_{\lambda\in\Lambda})$.

Let $\mathcal{A}=c_0(\{A_\lambda\}_{\lambda\in\Lambda})$ or $\mathcal{A}=\ell^\infty(\{\mathcal{A}_\lambda\}_{\lambda\in\Lambda})$.
Suppose that $\{\delta_\lambda\}_{\lambda\in\Lambda}$ is a family of bounded linear maps $\delta_\lambda\colon\mathcal{A}_\lambda\to\mathcal{X}_\lambda$ such that $\sup_{\lambda\in\Lambda}\|\delta_\lambda\|<\infty$.
Then their product is defined by
\begin{equation*}
  \prod\delta_\lambda\colon\mathcal{A}
  \to\ell^\infty(\{X_\lambda\}_{\lambda\in\Lambda}),\quad
  (A_\lambda)_{\lambda\in\Lambda}
  \mapsto(\delta_\lambda(A_\lambda))_{\lambda\in\Lambda}.
\end{equation*}
The following proposition shows that every derivation arises in the above way.

\begin{proposition}\label{prop derivation-decomposition}
Let $\{\mathcal{A}_\lambda\}_{\lambda\in\Lambda}$ be a family of Banach algebras and $\{\mathcal{X}_\lambda\}_{\lambda\in\Lambda}$ a family of Banach spaces such that $\mathcal{X}_\lambda$ is a Banach $\mathcal{A}_\lambda$-bimodule for each $\lambda\in\Lambda$.
Suppose that for any $X_\lambda\in\mathcal{X}_\lambda$, $\mathcal{A}_\lambda X_\lambda=\{0\}$ implies that $X_\lambda=0$.
Let $\mathcal{A}=c_0(\{A_\lambda\}_{\lambda\in\Lambda})$ or $\mathcal{A}=\ell^\infty(\{\mathcal{A}_\lambda\}_{\lambda\in\Lambda})$.

Then every derivation $\delta$ from $\mathcal{A}$ into $\ell^\infty(\{\mathcal{X}_\lambda\}_{\lambda\in\Lambda})$ is of the form $\prod\delta_\lambda$, where each $\delta_\lambda$ is a derivation from $\mathcal{A}_\lambda$ into $\mathcal{X}_\lambda$.
In particular, if $\delta$ is bounded then we have $\|\delta\|=\sup_{\lambda\in\Lambda}\|\delta_\lambda\|<\infty$.
\end{proposition}

\begin{proof}
For simplicity, each $\mathcal{A}_\lambda$ will be identified with a Banach subalgebra of $\mathcal{A}$.
For any $A\in\mathcal{A}_\lambda$ and $B\in\mathcal{A}_\mu$ with $\lambda\ne\mu$, we have
\begin{equation*}
  0=\delta(AB)=\delta(A)B+A\delta(B).
\end{equation*}
Note that $A\delta(B)\in\mathcal{X}_\lambda$ and $\delta(A)B\in\mathcal{X}_\mu$.
Hence $A\delta(B)=0$.
By assumption, the $\lambda$-coordinate of $\delta(B)$ is zero for each $\lambda\ne\mu$, i.e., $\delta(B)\in\mathcal{X}_\mu$.
Therefore, we can define a derivation by
\begin{equation*}
  \delta_\mu\colon\mathcal{A}_\mu\to\mathcal{X}_\mu,\quad
  \delta_\mu(B)=\delta(B).
\end{equation*}
Let $A=(A_\lambda)_{\lambda\in\Lambda}\in\mathcal{A}$ and $B\in\mathcal{A}_\mu$.
Since the $\mu$-coordinate of $A-A_\mu$ is zero, we have
\begin{equation*}
  0=\delta(B(A-A_\mu))=\delta_\mu(B)(A-A_\mu)+B\delta(A-A_\mu)=B\delta(A-A_\mu).
\end{equation*}
By assumption, the $\mu$-coordinate of $\delta(A-A_\mu)$ is zero, i.e., the $\mu$-coordinate of $\delta(A)$ is $\delta_\mu(A_\mu)$ for each $\mu\in\Lambda$.
Therefore, we have
\begin{equation*}
  \delta((A_\lambda)_{\lambda\in\Lambda})
  =(\delta_\lambda(A_\lambda))_{\lambda\in\Lambda}.
\end{equation*}
If $\delta$ is bounded, then one can obtain that $\|\delta\|=\sup_{\lambda\in\Lambda}\|\delta_\lambda\|$.
This completes the proof.
\end{proof}

\begin{remark}\label{rem derivation-decomposition}
Let $G$ be a directed graph with connected components $\{G_\lambda\}_{\lambda\in\Lambda}$.
Then
\begin{equation*}
  \mathcal{A}_G\cong c_0(\{\mathcal{A}_{G_\lambda}\}_{\lambda\in\Lambda}),\quad \mathfrak{L}_G\cong\ell^\infty(\{\mathfrak{L}_{G_\lambda}\}_{\lambda\in\Lambda}).
\end{equation*}
By \Cref{prop derivation-decomposition}, each bounded derivation $\delta$ from $\mathcal{A}_G$ into $\mathfrak{L}_G$ is of the form $\prod\delta_\lambda$ with $\|\delta\|=\sup_{\lambda\in\Lambda}\|\delta_\lambda\|$.
\end{remark}

\subsection{Locally inner derivations}
Let $\mathcal{A}$ be a Banach algebra, $\mathcal{X}$ a Banach $\mathcal{A}$-bimodule, and $\mathscr{A}$ a subset of $\mathcal{A}$.
A derivation $\delta\colon\mathcal{A}\to\mathcal{X}$ is called {\sl locally inner} with respect to $\mathscr{A}$ if there exists $T\in\mathcal{X}$ such that $\delta(A)=\delta_T(A)$ for all $A\in\mathscr{A}$.
It is clear that $\delta$ is inner if and only if $\delta$ is locally inner with respect to $\mathcal{A}$.
The following lemma will be used to study derivations from $\mathcal{A}_G$ into $\mathfrak{L}_G$ when $G$ is a finite directed graph.

\begin{lemma}\label{lem idempotent}
Let $\mathscr{A}$ be a finite set of idempotents $\{P_1,P_2,\ldots,P_n\}$ in $\mathcal{A}$ such that $\sum_{j=1}^nP_j=I$ and $P_iP_j=0$ for all $i\ne j$.
Then every derivation $\delta\colon\mathcal{A}\to\mathcal{X}$ is locally inner with respect to $\mathscr{A}$.
\end{lemma}

\begin{proof}
Let $T=\sum_{j=1}^nP_j\delta(P_j)$.
A direct calculation shows that
\begin{equation*}
  \delta_T(P_i)=P_i\delta(P_i)-\sum_{j=1}^nP_j\delta(P_j)P_i.
\end{equation*}
Since $P_i$ is an idempotent, we have $\delta(P_i)=\delta(P_i)P_i+P_i\delta(P_i)$.
It follows that $P_i\delta(P_i)P_i=0$.
Moreover, $0=\delta(P_jP_i)=\delta(P_j)P_i+P_j\delta(P_i)$ for all $j\ne i$.
Therefore,
\begin{equation*}
  \delta_T(P_i)=P_i\delta(P_i)+\sum_{j\ne i}^nP_j\delta(P_i)=\delta(P_i).
\end{equation*}
This completes the proof.
\end{proof}

The following result is very useful in computing the first cohomology groups.

\begin{proposition}\label{prop zero-at-A}
Let $\mathcal{A}$ be a Banach algebra, $\mathcal{X}$ a Banach $\mathcal{A}$-bimodule, and $\mathscr{A}$ a subset of $\mathcal{A}$.
Define
\begin{align*}
   \mathrm{Der}_{\mathscr{A}}(\mathcal{A},\mathcal{X}):&=\{\delta\in \mathrm{Der}(\mathcal{A},\mathcal{X})\colon
   \delta(A)=0~\text{for all}~A\in\mathscr{A}\},\\
   \mathrm{Inn}_{\mathscr{A}}(\mathcal{A},\mathcal{X}):&=\{\delta\in \mathrm{Inn}(\mathcal{A},\mathcal{X})\colon
   \delta(A)=0~\text{for all}~A\in\mathscr{A}\}.
\end{align*}
Suppose that every bounded derivation from $\mathcal{A}$ into $\mathcal{X}$ is locally inner with respect to $\mathscr{A}$.
Then $H^1(\mathcal{A},\mathcal{X})
\cong\mathrm{Der}_{\mathscr{A}}(\mathcal{A},\mathcal{X})
/\mathrm{Inn}_{\mathscr{A}}(\mathcal{A},\mathcal{X})$.
\end{proposition}

\begin{proof}
Define a linear map
\begin{equation*}
  \pi\colon\mathrm{Der}_{\mathscr{A}}(\mathcal{A},\mathcal{X})\to H^1(\mathcal{A},\mathcal{X}),\quad
  \delta\mapsto\delta+\mathrm{Inn}(\mathcal{A},\mathcal{X}).
\end{equation*}
By assumption, for any bounded derivation $\delta$ from $\mathcal{A}$ into $\mathcal{X}$, there exists $T\in\mathcal{X}$ such that $\delta(A)=\delta_T(A)$ for all $A\in\mathscr{A}$.
Since $\delta-\delta_T\in\mathrm{Der}_{\mathscr{A}}(\mathcal{A},\mathcal{X})$ and
\begin{equation*}
  \delta+\mathrm{Inn}(\mathcal{A},\mathcal{X})
  =(\delta-\delta_T)+\mathrm{Inn}(\mathcal{A},\mathcal{X}),
\end{equation*}
the map $\pi$ is onto.
Clearly, $\ker\pi=\mathrm{Inn}_{\mathscr{A}}(\mathcal{A},\mathcal{X})$.
This completes the proof.
\end{proof}

\subsection{Derivations on reduced subalgebras}
Using the following lemma, we can first investigate derivations of free semigroupoid algebras associated with finite directed graphs and then generalize to infinite directed graphs.
Recall that $L_F=\sum_{v\in F}L_v$ if $F$ is a subset of $\mathcal{V}$.

\begin{lemma}\label{lem reduced-derivation}
Let $G$ be a directed graph and $C>0$ a positive constant.
Suppose for any finite subset $F$ of $\mathcal{V}$ and any bounded derivation $\delta_F$ from $L_F\mathcal{A}_GL_F$ into $L_F\mathfrak{L}_GL_F$, there exists $T_F\in L_F\mathfrak{L}_GL_F$ such that $\delta_F=\delta_{T_F}$ and $\|T_F\|\leqslant C\|\delta_F\|$.
Then for any bounded derivation $\delta$ from $\mathcal{A}_G$ into $\mathfrak{L}_G$, there exists $T\in\mathfrak{L}_G$ such that $\delta=\delta_T$ and $\|T\|\leqslant C\|\delta\|$.
\end{lemma}

\begin{proof}
For every finite subset $F$ of $\mathcal{V}$, we define a derivation $\delta_F$ from $L_F\mathcal{A}_GL_F$ into $L_F\mathfrak{L}_GL_F$ by $\delta_F(A)=L_F\delta(A)L_F$.
By assumption, there exists an operator $T_F\in L_F\mathfrak{L}_GL_F$ such that $\delta_F=\delta_{T_F}$ and $\|T_F\|\leqslant C\|\delta_F\|\leqslant C\|\delta\|$.
By \Cref{prop compact}, the bounded net $\{T_F\}_F$ has a weak-operator limit point $T\in\mathfrak{L}_G$ with $\|T\|\leqslant C\|\delta\|$.
For any $p\in\mathbb{F}_G^+$ and any finite set $F\supseteq\{r(p),s(p)\}$, we have $L_p\in L_F\mathcal{A}_GL_F$ and hence
\begin{equation*}
  L_F\delta(L_p)L_F=\delta_F(L_p)=L_pT_F-T_FL_p.
\end{equation*}
It follows that $\delta(L_p)=\delta_T(L_p)$.
This completes the proof.
\end{proof}

\section{A weak Dixmier approximation theorem}\label{sec weak-Dixmier}
In this section, we prove a weak Dixmier approximation theorem for free semigroupoid algebras determined by strongly connected directed graphs, which will be used in \Cref{sec S-C} to investigate derivations.
At first, we need to characterize the center $Z(\mathfrak{L}_G)$ of $\mathfrak{L}_G$.

\subsection{Center}\label{subsec center}
Suppose $G$ is a directed graph.
The {\sl center} of $\mathfrak{L}_G$ is defined as
\begin{equation*}
  Z(\mathfrak{L}_G):
  =\{A\in\mathfrak{L}_G\colon AB=BA~\text{for all}~B\in\mathfrak{L}_G\}.
\end{equation*}
Recall that the $n$-circle graph $\mathscr{C}_n$ is a directed graph with $n$ distinct vertices $\{v_1,v_2,\ldots,v_n\}$ and edges $\{e_1,e_2,\ldots,e_n\}$ satisfying that
\begin{equation*}
  r(e_j)=v_j\quad\text{and}\quad s(e_j)=v_{j+1\,(\mathrm{mod}\,n)}
\end{equation*}
for all $1\leqslant j\leqslant n$.
Moreover, there is a unique minimal circle
\begin{equation}\label{equ c-j}
  c_j:=e_je_{j+1\,(\mathrm{mod}\,n)}\cdots e_{j+n-1\,(\mathrm{mod}\,n)}
\end{equation}
at the vertex $v_j$.

\begin{proposition}\label{prop center}
Suppose $G$ is a connected directed graph.
If $G$ is the $n$-circle graph $\mathscr{C}_n$, then each operator $A\in Z(\mathfrak{L}_G)$ has a Fourier expansion of the form
\begin{align*}
   \sum_{j=0}^{\infty}a_j\sum_{i=1}^{n}L_{c_i^j}.
\end{align*}
Otherwise, $Z(\mathfrak{L}_G)$ is trivial.
\end{proposition}

\begin{proof}
Let $A\in Z(\mathfrak{L}_G)$.
Then we can write
\begin{equation*}
  A=\sum_{v\in\mathcal{V}}L_vAL_v.
\end{equation*}
In particular, if $G=\mathscr{C}_n$, then $A=\sum_{i=1}^{n}L_{v_i}AL_{v_i}$.
Since $L_{v_i}AL_{v_i}$ commutes with $L_{c_i}$, it follows from \Cref{cor relative-commutant} that
\begin{equation*}
  L_{v_i}AL_{v_i}=\sum_{j=0}^{\infty}a_{ij}L_{c_i^j}.
\end{equation*}
Since $A$ commutes with $L_{e_i}$, we have $a_{ij}=a_{i+1\,(\mathrm{mod}\,n),j}$.
Hence $a_{ij}$ is independent of $1\leqslant i\leqslant n$, and we denote it by $a_j$.
Therefore, we have
\begin{equation*}
  A=\sum_{i=1}^{n}\sum_{j=0}^{\infty}a_jL_{c_i^j}.
\end{equation*}
Suppose $G$ is strongly connected and not an $n$-circle graph.
The conclusion is straightforward if $G$ is a vertex graph.
Otherwise, there are at least two distinct minimal circles at every vertex $v$. By \Cref{cor relative-commutant}, we have $L_vAL_v=\lambda_vL_v$, where $\lambda_v\in\mathbb{C}$.
For any $v_1,v_2\in\mathcal{V}$, there exists a path $p$ from $v_2$ to $v_1$.
Clearly, $L_pA=AL_p$ implies that $\lambda_{v_1}=\lambda_{v_2}$.
Therefore, we have $A=\lambda I$.

Next we assume that $G$ is not strongly connected.
Then there exists an edge $e\in\mathcal{E}$ such that there is no path from $y=r(e)$ to $x=s(e)$.
Since $A$ commutes with $L_e$, we have
\begin{equation*}
  L_yAL_y=\lambda L_y\quad\text{and}\quad L_xAL_x=\lambda L_x,
\end{equation*}
where $\lambda\in\mathbb{C}$.
Let $V_1$ be the set of all vertices such that $L_vAL_v=\lambda L_v$.
We claim that $V_1=V$.
Otherwise, since $G$ is connected, there exists an edge $f$ such that either $r(f)\in V_1,s(f)\in V\backslash V_1$ or $r(f)\in V\backslash V_1,s(f)\in V_1$.
Since $A$ commutes with $L_f$, in both cases we have
\begin{equation*}
  L_{r(f)}AL_{r(f)}=\lambda L_{r(f)}\quad\text{and}\quad
  L_{s(f)}AL_{s(f)}=\lambda L_{s(f)}.
\end{equation*}
That contradicts the maximality of $V_1$.
Therefore, $V_1=V$ and $A=\lambda I$.
We finish the proof.
\end{proof}

\begin{remark}\label{rem center}
Assume that $G=\mathscr{C}_n$ and $A$ is an operator in $Z(\mathfrak{L}_G)$ with Fourier expansion $\sum_{j=0}^{\infty}a_j\sum_{i=1}^{n}L_{c_i^j}$.
We define
\begin{equation*}
  \varphi_A(z)=\sum_{j=0}^{\infty}a_jz^j.
\end{equation*}
Then $\varphi\colon A\mapsto\varphi_A$ is a complete isometric isomorphism from $Z(\mathfrak{L}_G)$ onto $H^{\infty}(\mathbb{D})$, i.e., $Z(\mathfrak{L}_G)\cong H^{\infty}(\mathbb{D})$.
\end{remark}

\subsection{Weak Dixmier approximation theorem}
We denote by $\mathscr{D}=\mathscr{D}_G$ the set of all mappings $\alpha\colon B(\mathcal{H}_G)\to B(\mathcal{H}_G)$ of the form
\begin{align*}
    \alpha(A)=\frac{1}{n}\sum_{j=1}^nU_j^*AU_j,
\end{align*}
where every $U_j$ is an isometry in $\mathfrak{L}_G$.
Clearly, $\|\alpha\|=1$ and $\alpha|_{Z(\mathfrak{L}_G)}=\mathrm{id}$.
It is worth noting that $\alpha(\mathfrak{L}_G)$ is not a subset of $\mathfrak{L}_G$ in general.
For any $A\in\mathfrak{L}_G$, we define a subset $\mathscr{D}(A)$ of $B(\mathcal{H}_G)$ as
\begin{align*}
    \mathscr{D}(A):=\{\alpha(A)\colon\alpha\in\mathscr{D}\}.
\end{align*}
Let $\mathscr{D}(A)^-$ be the weak-operator closure of $\mathscr{D}(A)$.
Then we have the following weak Dixmier approximation theorem.

\begin{theorem}\label{thm weak-Dixmier}
Suppose $G$ is a strongly connected directed graph.
Then for any $A\in\mathfrak{L}_G$, we have
\begin{align*}
   \mathscr{D}(A)^-\cap Z(\mathfrak{L}_G)\neq\emptyset.
\end{align*}
In particular, $\mathscr{D}(A)\cap Z(\mathfrak{L}_G)\neq\emptyset$ if $G$ is an $n$-circle graph.
\end{theorem}

\begin{proof}
(i) Suppose $G$ is the $n$-circle $\mathscr{C}_n$.
For convenience, let $v_j=v_{j\,(\mathrm{mod}\,n)}$ and $e_j=e_{j\,(\mathrm{mod}\,n)}$ for all $j\in\mathbb{Z}$.
For every $\pi\in\{-1,1\}^{\mathbb{Z}/n\mathbb{Z}}$ and $1\leqslant i\leqslant n$, we define an isometry in $\mathfrak{L}_G$ by $W_{\pi,i}=\sum_{j=1}^n\pi(j)L_{e_je_{j+1}\cdots e_{j+i-1}}$.
Let
\begin{align*}
   \alpha(A)=\frac{1}{2^nn}\sum_{\pi,i}W_{\pi,i}^*AW_{\pi,i}.
\end{align*}
For any path $e_ae_{a+1}\cdots e_b\in\mathbb{F}_G^+$, it is straightforward to verify that
\begin{equation}\label{equ W*LW}
  W_{\pi,i}^*L_{e_ae_{a+1}\cdots e_b}W_{\pi,i}
  =\pi(a)\pi(b+1)L_{e_{a+i}e_{a+1+i}\cdots e_{b+i}}.
\end{equation}
If $b+1=a+kn$ for some $k\geqslant 1$, then $\pi(a)\pi(b+1)=1$ for every $\pi$.
It follows from \Cref{prop center} and \eqref{equ W*LW} that
\begin{equation*}
  \alpha(L_{e_ae_{a+1}\cdots e_b})=\frac{1}{n}\sum_{i=1}^{n}L_{c_i^k}\in Z(\mathfrak{L}_G).
\end{equation*}
If $b+1\not\equiv a\,(\mathrm{mod}\,n)$, then we have $\alpha(L_{e_ae_{a+1}\cdots e_b})=0$ by \eqref{equ W*LW}.
Moreover, for each vertex $v_j\in\mathcal{V}$, the equation \eqref{equ W*LW} becomes $W_{\pi,i}^*L_{v_j}W_{\pi,i}=L_{v_{j+i}}$.
It follows that $\alpha(L_{v_j})=\frac{1}{n}I\in Z(\mathfrak{L}_G)$.
Therefore, for any $A\in\mathfrak{L}_G$, we have
\begin{equation*}
  \alpha(A)\in Z(\mathfrak{L}_G).
\end{equation*}
In particular, $\mathscr{D}(A)^-\cap Z(\mathfrak{L}_G)\neq\emptyset$.

(ii) If $G$ is not an $n$-circle graph, then $Z(\mathfrak{L}_G)=\mathbb{C}I$ by \Cref{prop center}.
Suppose $F$ is a finite subset of $\mathcal{V}$ and $L_F=\sum_{v\in F}L_v$.
Let $U_{i,k,F}$ be the partial isometries given by \Cref{cor sum-U} and $V_{i,k,F}=I-L_F+U_{i,k,F}$.
We define
\begin{align*}
    \alpha_{k,F}(A)=\frac{1}{|F|}\sum_{i=1}^{|F|}V_{i,k,F}^*AV_{i,k,F}
    \in\mathscr{D}(A),\quad \lambda_F(A)=\frac{1}{|F|}\sum_{v\in F}\lambda_v(A),
\end{align*}
where $\lambda_v(A)$ is the Fourier coefficient of $A$ at $L_v$.
Then $|\lambda_F(A)|\leqslant\|A\|$.
It follows from \Cref{cor sum-U} that
\begin{align*}
    \text{\scriptsize {\rm WOT-}}\lim_{k\to\infty} \alpha_{k,F}(A)=(I-L_F)A(I-L_F)+\lambda_F(A)L_F.
\end{align*}
Since $\{\lambda_F(A)\}_F$ is a bounded net, it has a limit point $\lambda$ as $F\to\mathcal{V}$.
Then
\begin{align*}
    \text{\scriptsize {\rm WOT-}}\lim_{F\to\mathcal{V}}\lim_{k\to\infty} \alpha_{k,F}(A)=\lambda I.
\end{align*}
Hence $\lambda I\in\mathscr{D}(A)^-\cap Z(\mathfrak{L}_G)$.
This completes the proof.
\end{proof}

\section{Derivations and Strongly Connected Graphs}\label{sec S-C}
In this section, we show that for any bounded derivation $\delta$ from $\mathcal{A}_G$ into $\mathfrak{L}_G$, there exists $T\in\mathfrak{L}_G$ such that $\delta=\delta_T$ and $\|T\|\leqslant\|\delta\|$ if every connected component of $G$ is strongly connected.

\subsection{Derivations from $\mathcal{A}_G$ into $\mathfrak{L}_G$}
Let $\Lambda$ be an index set.
It was proved in \cite[Theorem 4.5]{HMW24} that for each bounded derivation $\delta$ from the non-commutative disc algebra $\mathcal{A}_{\Lambda}$ into the left regular semigroup algebra $\mathfrak{L}_{\Lambda}$, there exists $T\in\mathfrak{L}_{\Lambda}$ such that $\delta=\delta_T$.
The following is a reformulation of this result.

\begin{lemma}\label{thm 4.5}
Let $G$ be a directed graph with one vertex.
Then every bounded derivation $\delta$ from $\mathcal{A}_G$ into $\mathfrak{L}_G$ is an inner derivation.
\end{lemma}

Let $G$ be a directed graph.
For every vertex $v\in\mathcal{V}$, let $\mathcal{E}_v$ denote the set of all minimal circles at $v$.
It follows from \eqref{equ reduced-by-v} and \Cref{prop reduced} that $L_v\mathfrak{L}_GL_v$ and $L_v\mathcal{A}_GL_v$ are the left free semigroup algebra and the non-commutative disc algebra generated by $\mathcal{E}_v$, respectively.

\begin{proposition}\label{prop S-C-finite}
Suppose $G$ is a finite strongly connected directed graph.
Then every bounded derivation from $\mathcal{A}_G$ into $\mathfrak{L}_G$ is inner.
\end{proposition}

\begin{proof}
{\bf Step I.}
We claim that every bounded derivation $\delta$ from $\mathcal{A}_G$ into $\mathfrak{L}_G$ is locally inner with respect to $\mathscr{A}:=\bigcup_{v\in\mathcal{V}}L_v\mathcal{A}_GL_v$.
By \Cref{lem idempotent}, there exists $T\in\mathfrak{L}_G$ such that $\delta(L_v)=\delta_T(L_v)$ for each $v\in\mathcal{V}$.
Since $\delta$ is locally inner with respect to $\mathscr{A}$ if and only if $\delta-\delta_T$ is locally inner with respect to $\mathscr{A}$, we may assume that $\delta(L_v)=0$ for each $v\in\mathcal{V}$.
Since $\delta(L_v)=0$, we have
\begin{equation*}
  \delta(A)=\delta(L_vAL_v)=L_v\delta(A)L_v\in L_v\mathfrak{L}_GL_v
\end{equation*}
for all $A\in L_v\mathcal{A}_GL_v$.
By \Cref{thm 4.5}, there exists $T_v\in L_v\mathfrak{L}_GL_v$ such that $\delta(A)=\delta_{T_v}(A)$ for all $A\in L_v\mathcal{A}_GL_v$.
Let $T=\sum_{v\in\mathcal{V}}T_v$.
It is routine to verify that $\delta(A)=\delta_T(A)$ for all $A\in\mathscr{A}$.

{\bf Step II.}
Since $G$ is a strongly connected, there exists a circle $e_1e_2\cdots e_n\in\mathbb{F}_G^+$ such that every vertex in $\mathcal{V}$ is of the form $r(e_j)$ for some $1\leqslant j\leqslant n$.
Let
\begin{equation*}
  v_j=r(e_j)~\text{and}~
  \mathscr{A}_j=\{L_{e_1},L_{e_2},\ldots,L_{e_j}\}\cup\mathscr{A}
\end{equation*}
for $1\leqslant j\leqslant n$, and $v_{n+1}=s(e_n)=r(e_1)=v_1$.
Next, we claim that every bounded derivation from $\mathcal{A}_G$ into $\mathfrak{L}_G$ is locally inner with respect to each $\mathscr{A}_j$.
Without loss of generality, let $\delta$ be a bounded derivation from $\mathcal{A}_G$ into $\mathfrak{L}_G$ such that $\delta|_{\mathscr{A}}=0$.
If $v_2=v_1$, then $L_{e_1}\in\mathscr{A}$ and hence $\delta(L_{e_1})=0$.
Assume that $v_2\ne v_1$.
Let $p$ be a path from $v_1$ to $v_2$.
Then $\delta(L_{e_1})=L_{v_1}\delta(L_{e_1})L_{v_2}$ and $\delta(L_p)=L_{v_2}\delta(L_p)L_{v_1}$.
Since $L_pL_{e_1}\in\mathscr{A}$, we have
\begin{equation*}
  0=\delta(L_pL_{e_1})=\delta(L_p)L_{e_1}+L_p\delta(L_{e_1}).
\end{equation*}
By comparing the Fourier coefficients, there exists $T_1\in L_{v_1}\mathfrak{L}_GL_{v_1}$ independent of the choice of $p$ such that $\delta(L_p)=L_pT_1$ and $\delta(L_{e_1})=-T_1L_{e_1}$.
By the fact $L_{e_1}L_p\in\mathscr{A}$, we also have $\delta(L_{e_1}L_p)=0$.
Therefore, $T_1\in\{L_{e_1}L_p\}'\cap L_{v_1}\mathfrak{L}_GL_{v_1}$ for every path $p$ from $v_1$ to $v_2$.
If there are two distinct minimal circles at $v_1$, then we must have $T_1=cL_{v_1}$ for some $c\in\mathbb{C}$.
Otherwise, $L_{v_1}\mathfrak{L}_GL_{v_1}$ is commutative.
In both case, $\delta$ and $\delta_{T_1}$ coincide on $\mathscr{A}_1$.

Suppose every bounded derivation from $\mathcal{A}_G$ into $\mathfrak{L}_G$ is locally inner with respect to $\mathscr{A}_{j-1}$ for some $2\leqslant j\leqslant n$.
Without loss of generality, let $\delta$ be a bounded derivation from $\mathcal{A}_G$ into $\mathfrak{L}_G$ such that $\delta|_{\mathscr{A}_{j-1}}=0$.
If $v_{j+1}=v_i$ for some $1\leqslant i\leqslant j$, then $L_{e_i}\cdots L_{e_{j-1}}L_{e_j}\in\mathscr{A}$.
It follows that $L_{e_i}\cdots L_{e_{j-1}}\delta(L_{e_j})=0$.
Combining with $\delta(L_{e_j})=L_{v_j}\delta(L_{e_j})L_{v_{j+1}}$, we have $\delta(L_{e_j})=0$.
Assume that $v_{j+1}\ne v_i$ for every $1\leqslant i\leqslant j$.
Let $p$ be a path from $v_j$ to $v_{j+1}$.
Then $L_pL_{e_j},L_{e_j}L_p\in\mathscr{A}$.
Similarly, there exists $T_j\in L_{v_j}\mathfrak{L}_GL_{v_j}$ such that $\delta$ and $\delta_{T_j}$ coincide on $\mathscr{A}_j$.

{\bf Step III.}
We claim that every bounded derivation from $\mathcal{A}_G$ into $\mathfrak{L}_G$ is inner.
Without loss of generality, let $\delta$ be a bounded derivation from $\mathcal{A}_G$ into $\mathfrak{L}_G$ such that $\delta|_{\mathscr{A}_n}=0$.
Recall that $e_1e_2\cdots e_n\in\mathbb{F}_G^+$ and every vertex in $\mathcal{V}$ is of the form $r(e_j)$ for some $1\leqslant j\leqslant n$.
For any edge $e\in\mathcal{E}$, there exists a path $p$ of the form $e_{j_1}e_{j_2}\cdots e_{j_k}$ such that $r(p)=s(e)$ and $s(p)=r(e)$.
Since $L_pL_e\in\mathscr{A}$, we have $L_p\delta(L_e)=0$.
It follows that $\delta(L_e)=0$ for all $e\in\mathcal{E}$.
Therefore, we have $\delta=0$.
This completes the proof.
\end{proof}

The following corollary is an application of \Cref{thm weak-Dixmier} and \Cref{prop S-C-finite}.

\begin{corollary}\label{cor S-C-finite}
Suppose $G$ is a finite strongly connected directed graph.
Then for every bounded derivation $\delta$ from $\mathcal{A}_G$ into $\mathfrak{L}_G$, there exists $T\in\mathfrak{L}_G$ such that $\delta=\delta_T$ and $\|T\|\leqslant\|\delta\|$.
\end{corollary}

\begin{proof}
By \Cref{prop S-C-finite}, there exists $S\in\mathfrak{L}_G$ such that $\delta=\delta_S$.
Recall that every $\alpha\in\mathscr{D}$ is of the form
\begin{align*}
    \alpha(A)=\frac{1}{n}\sum_{j=1}^nU_j^*AU_j,
\end{align*}
where every $U_j$ is an isometry in $\mathfrak{L}_G$.
It follows that
\begin{equation*}
  \|S-\alpha(S)\|\leqslant\frac{1}{n}\sum_{j=1}^n\|S-U_j^*SU_j\|
  =\frac{1}{n}\sum_{j=1}^n\|U_j^*\delta(U_j)\|\leqslant\|\delta\|.
\end{equation*}
Let $C\in\mathscr{D}(S)^-\cap Z(\mathfrak{L}_G)$ by \Cref{thm weak-Dixmier}.
Then $\|S-C\|\leqslant\|\delta\|$.
We complete the proof by taking $T=S-C$.
\end{proof}

Now we are able to extend \Cref{cor S-C-finite} to infinite case.
The following is the main result of this section.

\begin{theorem}\label{thm S-C}
Suppose $G$ is a strongly connected directed graph.
Then for every bounded derivation $\delta$ from $\mathcal{A}_G$ into $\mathfrak{L}_G$, there exists $T\in\mathfrak{L}_G$ such that $\delta=\delta_T$ and $\|T\|\leqslant\|\delta\|$.
\end{theorem}

\begin{proof}
Let $F$ be a finite subset of $\mathcal{V}$, $G_F$ the reduced graph of $G$ by $F$, and $L_F=\sum_{v\in F}L_v$.
It follows from \Cref{prop reduced} that
\begin{equation*}
  L_F\mathcal{A}_GL_F\cong\mathcal{A}_{G_F},\quad
  L_F\mathfrak{L}_GL_F\cong\mathfrak{L}_{G_F}.
\end{equation*}
Note that $G_F$ is a finite strongly connected directed graph.
By \Cref{cor S-C-finite}, for any bounded derivation $\delta_F$ from $L_F\mathcal{A}_GL_F$ into $L_F\mathfrak{L}_GL_F$, there exists an operator $T_F\in L_F\mathfrak{L}_GL_F$ such that $\delta_F=\delta_{T_F}$ and $\|T_F\|\leqslant C\|\delta_F\|$.
We complete the proof by applying \Cref{lem reduced-derivation}.
\end{proof}

\begin{corollary}\label{cor S-C}
Suppose $G$ is a directed graph such that every connected component of $G$ is strongly connected.
Then for every bounded derivation $\delta$ from $\mathcal{A}_G$ into $\mathfrak{L}_G$, there exists $T\in\mathfrak{L}_G$ such that $\delta=\delta_T$ and $\|T\|\leqslant\|\delta\|$.
\end{corollary}

\begin{proof}
It follows from \Cref{rem derivation-decomposition} and \Cref{thm S-C}.
\end{proof}

\subsection{Derivations on $\mathcal{A}_G$}
By the proof of \cite[Lemma 5.2]{KP04}, $\mathcal{A}_G$ is semisimple if every connected component of $G$ is strongly connected.
In \cite{JS68}, Johnson and Sinclair proved that every derivation on a semisimple Banach algebra is bounded.
Since $\mathcal{A}_G$ is a Banach subalgebra of $\mathfrak{L}_G$, we can immediately obtain the following corollary by \Cref{cor S-C}.

\begin{corollary}\label{cor S-C-2}
Suppose $G$ is a directed graph such that every connected component of $G$ is strongly connected.
Then for every  derivation $\delta$ on $\mathcal{A}_G$, there exists $T\in\mathfrak{L}_G$ such that $\delta=\delta_T$ and $\|T\|\leqslant\|\delta\|$.
\end{corollary}

Note that it seems not easy to determine whether $T\in\mathcal{A}_G$ in \Cref{cor S-C-2}.
By the following lemma, we only need to focus on finite directed graphs.

\begin{lemma}\label{lem infinite-not-inner}
Let $G$ be a connected directed graph with infinite vertices.
Then there exists a bounded derivation on $\mathcal{A}_G$ which is not inner.
\end{lemma}

\begin{proof}
Let $F$ be a subset of $\mathcal{V}$ such that both $F$ and $\mathcal{V}\backslash F$ are infinite sets, $L_F=\sum_{v\in F}L_v$, and $\delta$ a bounded derivation given by $\delta(A)=AL_F-L_FA$.
For any path $p\in\mathbb{F}_G^+$, we have
\begin{equation*}
  \delta(L_p)
  =\begin{cases}
     -L_p, & \mbox{if } r(p)\in F,s(p)\in\mathcal{V}\backslash F,\\
     L_p, & \mbox{if } s(p)\in F,r(p)\in\mathcal{V}\backslash F,\\
     0, & \mbox{otherwise}.
   \end{cases}
\end{equation*}
It follows that $\delta(L_p)\in\mathcal{A}_G$.
Therefore, $\delta$ is a derivation on $\mathcal{A}_G$.
If there exists $T\in\mathcal{A}_G$ such that $\delta=\delta_T$, then $T-L_F\in Z(\mathfrak{L}_G)$.
By \Cref{prop center}, there exists $\lambda\in\mathbb{C}$ such that $T=L_F+\lambda I\in\mathcal{A}_G$.
That is a contradiction.
\end{proof}

The following technique lemma is useful when studying derivations on $\mathcal{A}_G$.

\begin{lemma}\label{lem Lp*}
Let $G$ be a directed graph.
Then for every $p\in\mathbb{F}_G^+$, we have
\begin{equation*}
  L_p^*\mathcal{A}_G\cap\mathfrak{L}_G\subseteq\mathcal{A}_G.
\end{equation*}
\end{lemma}

\begin{proof}
If $p$ is a vertex, then $L_p^*=L_p$ and the conclusion is clear.
We assume that $p=e_1e_2\cdots e_n$ has length $n\geqslant 1$.
For $1\leqslant j\leqslant n$, let
\begin{equation*}
  p_j=e_1e_2\cdots e_{j-1},\quad q_j=e_je_{j+1}\cdots e_n.
\end{equation*}
Then $p=p_jq_j$ and $p_1=r(p)$.
Let $T\in\mathcal{A}_G$ such that $L_p^*T\in\mathfrak{L}_G$.
For any $\varepsilon>0$, there exists an operator $A\in\mathrm{span}\{L_w\colon w\in\mathbb{F}_G^+\}$ such that $\|T-A\|<\varepsilon$.
Let $\lambda_j$ and $\mu_j$ be the Fourier coefficients of $T$ and $A$ at $L_{p_j}$, respectively.
It follows from $L_p^*T\in\mathfrak{L}_G$ that $T_0:=\sum_{j=1}^{n}\lambda_jL_{q_j}^*\in\mathfrak{L}_G$.
Since $T_0\xi_v=0$ for every vertex $v\in\mathcal{V}$, we have $T_0=0$ and hence $\lambda_j=0$ for every $1\leqslant j\leqslant n$.
In particular, we have $|\mu_j|\leqslant\|T-A\|<\varepsilon$.
It is clear that
\begin{equation*}
  B:=L_p^*A-\sum_{j=1}^{n}\mu_jL_{q_j}^*
  \in\mathrm{span}\{L_w\colon w\in\mathbb{F}_G^+\}.
\end{equation*}
Thus, we have
\begin{equation*}
  \|L_p^*T-B\|=\left\|L_p^*(T-A)+\sum_{j=1}^{n}\mu_jL_{q_j}^*\right\|
  <(n+1)\varepsilon.
\end{equation*}
This completes the proof.
\end{proof}

\begin{proposition}\label{prop S-C-uncountable}
Let $G$ be a finite strongly connected directed graph with uncountably many edges.
Then every derivation on $\mathcal{A}_G$ is inner.
\end{proposition}

\begin{proof}
Let $\delta$ be a derivation on $\mathcal{A}_G$.
By the proof of \Cref{lem idempotent}, there exists $S\in\mathcal{A}_G$ such that $\delta(L_v)=\delta_S(L_v)$ for all $v\in\mathcal{V}$.
We may assume that $\delta(L_v)=0$ for all $v\in\mathcal{V}$.
By \Cref{cor S-C-2}, there exists $T\in\mathfrak{L}_G$ such that $\delta=\delta_T$.
Then $TL_v=L_vT$.
We only need to show that $T\in\mathcal{A}_G$.

For any $v\in\mathcal{V}$, there are at most countably many nonzero terms in the Fourier expansion of $L_vTL_v$.
By assumption, there are uncountably many minimal circles at $v$.
Since $L_{p_1}^*L_{p_2}=0$ for any two distinct minimal circles $p_1$ and $p_2$ at $v$, there exists a minimal circle $p$ at $v$ such that $L_p^*T=\lambda_vL_p^*$, where $\lambda_v$ is the Fourier coefficient at $L_v$.
It follows that
\begin{equation*}
  L_p^*\delta(L_p)=L_vT-\lambda_vL_v\in L_p^*\mathcal{A}_G\cap\mathfrak{L}_G.
\end{equation*}
Hence $L_vT\in\mathcal{A}_G$ by \Cref{lem Lp*}.
Therefore, $T=\sum_{v\in\mathcal{V}}L_vT\in\mathcal{A}_G$.
\end{proof}

\begin{corollary}
Let $G$ be a finite strongly connected directed graph with uncountably many edges.
Then for every derivation $\delta$ on $\mathcal{A}_G$, there exists $T\in\mathcal{A}_G$ such that $\delta=\delta_T$ and $\|T\|\leqslant\|\delta\|$.
\end{corollary}

\begin{proof}
By \Cref{cor S-C-2}, there exists $T\in\mathfrak{L}_G$ such that $\delta=\delta_T$ and $\|T\|\leqslant\|\delta\|$.
By \Cref{prop S-C-uncountable}, there exists $S\in\mathcal{A}_G$ such that $\delta=\delta_S$.
It follows that $T-S\in Z(\mathfrak{L}_G)$.
By applying \Cref{prop center}, we have $T=S+\lambda I$ for some scalar $\lambda\in\mathbb{C}$.
Therefore, $T\in\mathcal{A}_G$.
\end{proof}

For non-commutative disc algebras, we provide a simple criterion for \Cref{ques inner} in the following lemma.

\begin{lemma}\label{lemma condition-inner}
Suppose $G$ is a directed graph with one vertex and $N$ loops $\{e_j\}_{j=1}^{N}$, where $2\leqslant N\leqslant\infty$.
Let $\delta$ be a derivation on $\mathcal{A}_G$.
If $\delta$ is locally inner at $L_{e_1}$, then $\delta$ is an inner derivation.
\end{lemma}

\begin{proof}
We may assume that $\delta(L_{e_1})=0$.
By \Cref{prop S-C-finite}, there exists an operator $T\in\mathfrak{L}_G$ such that $\delta=\delta_T$.
Since $T\in\{L_{e_1}\}'\cap\mathfrak{L}_G$, the Fourier expansion of $T$ is of the form $\sum_{n=0}^{\infty}a_nL_{e_1^n}$ by \Cref{cor relative-commutant}.
It follows that $T-a_0I=L_{e_2}^*\delta(L_{e_2})\in\mathcal{A}_G$ by \Cref{lem Lp*}.
Therefore, $T\in\mathcal{A}_G$.
This completes the proof.
\end{proof}

We consider derivations of an $n$-circle graph in the next proposition.
Our approach is quite different from \cite[Corollary 2]{Dun07}.

\begin{proposition}\label{prop circle-inner}
Let $G$ be the $n$-circle graph $\mathscr{C}_n$ for some $n\geqslant 1$.
Then every derivation on $\mathcal{A}_G$ is inner.
\end{proposition}

\begin{proof}
We will adopt the notation introduced in \Cref{subsec center}.
Similar to the proof of \Cref{prop S-C-uncountable}, we may assume that $\delta(L_{v_j})=0$ for all $1\leqslant j\leqslant n$, and $T$ is an operator in $\mathfrak{L}_G$ such that $\delta=\delta_T$.
Since $\delta(L_{e_1})\in L_{v_1}\mathcal{A}_GL_{v_2}$, there exists $A_2\in L_{v_2}\mathfrak{L}_GL_{v_2}$ such that $\delta(L_{e_1})=L_{e_1}A_2$.
By \Cref{lem Lp*}, we have $A_2=L_{e_1}^*\delta(L_{e_1})\in L_{v_2}\mathcal{A}_GL_{v_2}$.
Similarly, for every $1\leqslant j\leqslant n-1$, there exists an operator $A_{j+1}\in L_{v_{j+1}}\mathcal{A}_GL_{v_{j+1}}$ such that
\begin{equation*}
  \delta(L_{e_1e_2\cdots e_j})=L_{e_1e_2\cdots e_j}A_{j+1}.
\end{equation*}
Let $A_1=0$.
Then $\delta(L_{e_j})=L_{e_j}A_{j+1}-A_jL_{e_j}$ for every $1\leqslant j\leqslant n-1$.
Define
\begin{equation*}
  A=\sum_{j=1}^{n}A_j\in\mathcal{A}_G.
\end{equation*}
It is clear that $\delta(L_{e_j})=\delta_A(L_{e_j})$ for $1\leqslant j\leqslant n-1$.
Let $\delta'=\delta-\delta_A$.
Since $L_{e_1\cdots e_{n-1}e_n}\in L_{v_1}\mathcal{A}_GL_{v_1}$ and $L_{v_1}\mathcal{A}_GL_{v_1}$ is commutative, we have
\begin{equation*}
  0=\delta'(L_{e_1\cdots e_{n-1}e_n})=L_{e_1\cdots e_{n-1}}\delta'(L_{e_n})=0.
\end{equation*}
It follows that $\delta'(L_{e_n})=0$.
Therefore, $\delta'=0$ and hence $\delta=\delta_A$.
\end{proof}

\begin{corollary}\label{cor circle-inner}
Let $G$ be the $n$-circle graph $\mathscr{C}_n$ for some $n\geqslant 1$.
Then for every derivation $\delta$ on $\mathcal{A}_G$, there exists $T\in\mathcal{A}_G$ such that $\delta=\delta_T$ and $\|T\|\leqslant\|\delta\|$.
\end{corollary}

\begin{proof}
By \Cref{prop circle-inner}, there exists $S\in\mathcal{A}_G$ such that $\delta=\delta_S$.
By the proof of \Cref{thm weak-Dixmier} and \Cref{cor S-C-finite}, there exists $\alpha\in\mathscr{D}$ such that
\begin{equation*}
  \alpha(S)\in Z(\mathfrak{L}_G)\cap\mathcal{A}_G,\quad
  \|S-\alpha(S)\|\leqslant\|\delta\|.
\end{equation*}
We can take $T=S-\alpha(S)$.
\end{proof}

\subsection{Derivations on $\mathfrak{L}_G$}
Next we consider derivations on $\mathfrak{L}_G$.
For the $n$-circle graph $\mathscr{C}_n$, we can view $\mathfrak{L}_{\mathscr{C}_n}$ as a matrix function algebra of the form:
\begin{align*}
    \begin{pmatrix}
    f_{11}(z^n)&zf_{12}(z^n)&z^2f_{13}(z^n)&\cdots&z^{n-1}f_{1n}(z^n)\\
    z^{n-1}f_{21}(z^n)&f_{22}(z^n)&zf_{23}(z^n)&\cdots&z^{n-2}f_{2n}(z^n)\\
    z^{n-2}f_{31}(z^n)&z^{n-1}f_{32}(z^n)&f_{33}(z^n)&\cdots&
    z^{n-3}f_{3n}(z^n)\\
    \vdots&\vdots&\vdots&\ddots&\vdots\\
    zf_{n1}(z^n)&z^2f_{n2}(z^n)&z^3f_{n3}(z^n)&\cdots&f_{nn}(z^n)
\end{pmatrix},
\end{align*}
where $f_{ij}\in H^{\infty}(\mathbb{D})$ for all $1\leq i,j\leq n$.
The next proposition provides an affirmative answer to the question proposed at the end of \cite{Dun07}.

\begin{proposition}\label{prop circle-inner-2}
Suppose $G$ is a directed graph with connected components $\{\mathscr{C}_{n_\lambda}\}_{\lambda\in\Lambda}$.
Then for every derivation $\delta$ on $\mathfrak{L}_G$, there exists $T\in\mathfrak{L}_G$ such that $\delta=\delta_T$ and $\|T\|\leqslant\|\delta\|$.
\end{proposition}

\begin{proof}
We may assume that $G=\mathscr{C}_n$ by \Cref{rem derivation-decomposition} and we will adopt the notation introduced in \Cref{subsec center}.
By \Cref{thm S-C}, there exists $T\in\mathfrak{L}_G$ such that $\delta|_{\mathcal{A}_G}=\delta_T|_{\mathcal{A}_G}$ and $\|T\|\leqslant\|\delta\|$.
Let $\delta'=\delta-\delta_T$ and $\mathfrak{L}_{ij}=L_{v_i}\mathfrak{L}_{\mathscr{C}_n}L_{v_j}$ for every $1\leqslant i,j\leqslant n$.
Since $\delta'(L_{v_i})=0$, $\delta'$ is a derivation on the commutative semisimple Banach algebra $\mathfrak{L}_{jj}$ for every $1\leqslant j\leqslant n$.
It follows from \cite{SW55} that $\delta'|_{\mathfrak{L}_{jj}}=0$.
Let $p_{ji}$ be a path from $v_i$ to $v_j$.
For any $A\in\mathfrak{L}_{ij}$, we have $AL_{p_{ji}}\in\mathfrak{L}_{ii}$.
Since $L_{p_{ji}}\in\mathcal{A}_G$, we have $\delta'(L_{p_{ji}})=0$ and
\begin{equation*}
  0=\delta'(AL_{p_{ji}})=\delta'(A)L_{p_{ji}}+A\delta'(L_{p_{ji}})
  =\delta'(A)L_{p_{ji}}.
\end{equation*}
It follows that $\delta'(A)=0$ and hence $\delta'|_{\mathfrak{L}_{ij}}=0$.
Note that $\mathfrak{L}_G=\sum_{i,j=1}^{n}\mathfrak{L}_{ij}$.
Therefore, $\delta'=0$ and $\delta=\delta_T$.
This completes the proof.
\end{proof}

\section{Derivations and Generalized fruit trees}\label{sec fruit-tree}
Let $G$ be a finite connected directed graph which is not strongly connected.
In this section, we prove that every bounded derivation from $\mathcal{A}_G$ into $\mathfrak{L}_G$ is inner if and only if $G$ is a fruit tree (see \Cref{thm fruit-tree} and \Cref{thm converse}).
We introduce the alternating number to handle the infinite case in \Cref{subsec alternating}.

\subsection{Fruit trees}
Suppose $G=(\mathcal{V},\mathcal{E},r,s)$ is a directed graph.
For each $v\in\mathcal{V}$, its {\sl out-degree} and {\sl in-degree} are defined by
\begin{equation*}
  \deg^+(v)=|\{e\in\mathcal{E}\colon s(e)=v\}|\quad\text{and}\quad
  \deg^-(v)=|\{e\in\mathcal{E}\colon r(e)=v\}|,
\end{equation*}
respectively.
The {\sl degree} of a vertex is the number of edges connecting with it, i.e., $\deg(v)=\deg^+(v)+\deg^-(v)$.

A {\sl polygon} of length $n\geqslant 1$ in $G$ is an undirected circle, i.e., a finite set $\{v_j,e_j\}_{j=1}^{n}$ such that $\{v_j\}_{j=1}^{n}$ is a set of distinct vertices and $\{e_j\}_{j=1}^{n}$ is a set of edges such that
\begin{equation*}
  \{r(e_j),s(e_j)\}=\{v_j,v_{j+1\,(\mathrm{mod}\,n)}\}
\end{equation*}
for each $1\leqslant j\leqslant n$.

\begin{definition}\label{def tree}
A connected directed graph $G$ with $|\mathcal{V}|\geqslant 2$ is called a {\sl generalized tree} if it satisfies one of the following equivalent conditions:
\begin{enumerate}[(i)]
\item $G$ contains no polygon;
\item $G$ becomes disconnected if any edge is removed from $G$.
\end{enumerate}
If $G$ is a generalized tree with finitely many vertices, then $G$ is called a {\sl tree}.
\end{definition}

For simplicity, we assume that every tree has at least two vertices in \Cref{def tree}.
A vertex of a generalized tree with degree one is called a {\sl leaf}.
It is well-known that a finite connected directed graph $G$ is a tree if and only if $|\mathcal{V}|=|\mathcal{E}|+1$.
For any tree $G$, the free semigroupoid algebra $\mathfrak{L}_G$ is finite-dimensional.
Hence $\mathcal{A}_G=\mathfrak{L}_G$.

\begin{lemma}\label{lem tree}
Let $G$ be a tree.
Then every derivation on $\mathcal{A}_G$ is inner.
\end{lemma}

\begin{proof}
Let $n=|\mathcal{E}|\geqslant 1$.
Since $G$ is connected, we can write $\mathcal{V}=\{v_0,v_1,\ldots,v_n\}$ such that there exists an edge $e_j$ between $v_j$ and $\{v_0,v_1,\ldots,v_{j-1}\}$ for every $1\leqslant j\leqslant n$.
Clearly, we have $\mathcal{E}=\{e_1,e_2,\ldots,e_n\}$.

The proof strategy we employ closely follows that of \Cref{prop S-C-finite}.
Let
\begin{equation*}
  \mathscr{A}=\{L_{v_0},L_{v_1},\ldots,L_{v_n}\}\quad\text{and}\quad
  \mathscr{A}_j=\{L_{e_1},L_{e_2}\ldots,L_{e_j}\}\cup\mathscr{A}
\end{equation*}
for $1\leqslant j\leqslant n$.
We claim that every derivation $\delta$ on $\mathcal{A}_G$ is locally inner with respect to each $\mathscr{A}_j$.
By applying \Cref{lem idempotent}, we may assume that $\delta|_{\mathscr{A}}=0$.
Since $\delta(L_{e_1})\in L_{r(e_1)}\mathcal{A}_GL_{s(e_1)}=\mathbb{C}L_{e_1}$, we have $\delta(L_{e_1})=\lambda_1L_{e_1}$.
We define
\begin{equation*}
  T_1=
  \begin{cases}
    -\lambda_1L_{v_1}, & \mbox{if } r(e_1)=v_1,\\
    \lambda_1L_{v_1}, & \mbox{if } s(e_1)=v_1.
  \end{cases}
\end{equation*}
Then $\delta$ and $\delta_{T_1}$ coincide on $\mathscr{A}_1$.

Suppose every derivation on $\mathcal{A}_G$ is locally inner with respect to $\mathscr{A}_{j-1}$ for some $2\leqslant j\leqslant n$.
Then we may assume that $\delta|_{\mathscr{A}_{j-1}}=0$ and $\delta(L_{e_j})=\lambda_jL_{e_j}$.
Let
\begin{equation*}
  T_j=
  \begin{cases}
    -\lambda_jL_{v_j}, & \mbox{if } r(e_j)=v_j,\\
    \lambda_jL_{v_j}, & \mbox{if } s(e_j)=v_j.
  \end{cases}
\end{equation*}
Then $\delta$ and $\delta_{T_j}$ coincide on $\mathscr{A}_j$.
This completes the proof.
\end{proof}

Let $G_t=(\mathcal{V}_t,\mathcal{E}_t,s,r)$ and $G_f=(\mathcal{V}_f,\mathcal{E}_f,s,r)$ be directed graphs with a set of identified vertices
\begin{equation*}
  V\subseteq\mathcal{V}_t\cap\mathcal{V}_f.
\end{equation*}
Then their {\sl amalgamated graph} $G=G_t\sqcup_V G_f$ is the directed graph with vertices $\mathcal{V}=\mathcal{V}_t\sqcup_V \mathcal{V}_f$ and edges $\mathcal{E}=\mathcal{E}_t\sqcup\mathcal{E}_f$.

\begin{definition}[\bf Generalized fruit tree]\label{def fruit-tree}
Suppose $G_t$ is a generalized tree, $V:=\{v_\lambda\}_{\lambda\in\Lambda}$ is a subset of leaves of $G_t$, and $e_\lambda$ is the edge between $v_\lambda$ and $\mathcal{V}_t\backslash\{v_\lambda\}$ for each $\lambda\in\Lambda$.
Suppose $G_f$ is a directed graph with connected components $\{\mathscr{C}_{n_\lambda}\}_{\lambda\in\Lambda}$, and each $v_\lambda$ is identified with a vertex of the $n_\lambda$-circle graph $\mathscr{C}_{n_\lambda}$.

If $V\ne\mathcal{V}_t$, then the amalgamated graph $G_t\sqcup_V G_f$ is called a {\sl generalized fruit tree} with $|\Lambda|$ fruits.
If $\deg_{G_t}^+(v_\lambda)=1$, i.e., $s(e_\lambda)=v_\lambda$, then $\mathscr{C}_{n_\lambda}$ is called an {\sl out-fruit}.
Otherwise, $\mathscr{C}_{n_\lambda}$ is called an {\sl in-fruit}.

If $G$ is a generalized fruit tree with finitely many vertices, then $G$ is called a {\sl fruit tree}.
\end{definition}

If $\Lambda$ is empty, then $G_t\sqcup_V G_f=G_t$ is a generalized tree, which is a fruit tree with no fruit.
If $|\mathcal{V}_t|=2$, then $|V|=0$ or $|V|=1$.
If $|\mathcal{V}_t|\geqslant 3$, then the condition $V\ne\mathcal{V}_t$ holds automatically.
By definition, every generalized fruit tree is not strongly connected.

By applying \Cref{prop reduced}, we are able to compute the norms of some special operators in $\mathfrak{L}_G$.

Let $A(\mathbb{D})$ be the function space consisting of all continuous functions $f$ on $\overline{\mathbb{D}}$ such that $f$ is analytic in $\mathbb{D}$.
\begin{lemma}\label{lem fruit-tree-norm}
Suppose $G=G_t\sqcup_VG_f$ is a generalized fruit tree, $\mathscr{C}_{n_\lambda}$ is an in-fruit, $e_\lambda$ is the edge with $r(e_\lambda)=v_\lambda$ and $s(e_\lambda)\in\mathcal{V}_t\backslash V$, and $w_\lambda$ is the unique minimal circle at $v_\lambda$.
Let $f(z)=\sum_{n=0}^{\infty}a_nz^n$.
\begin{enumerate}[$(i)$]
\item Let $A=f(L_{w_\lambda})\in\mathfrak{L}_G$, i.e., $A$ has a Fourier expansion $\sum_{n=0}^{\infty}a_nL_{w_\lambda^n}$.
    Then $\|A\|=\|f\|_{H^\infty(\mathbb{D})}$.
    Furthermore, $A\in\mathcal{A}_G$ if and only if $f\in A(\mathbb{D})$.
\item Let $A=f(L_{w_\lambda})L_{e_\lambda}\in\mathfrak{L}_G$, i.e., $A$ has a Fourier expansion $\sum_{n=0}^{\infty}a_nL_{w_\lambda^ne_\lambda}$.
Then $\|A\|=\|f\|_{H^2(\mathbb{D})}$ and $A\in\mathcal{A}_G$.
\end{enumerate}
\end{lemma}

\begin{proof}
(i) We only need to take $F=\{v_\lambda\}$ in \Cref{prop reduced}.

(ii) If we take $F=\{v_\lambda,s(e_\lambda)\}$ in \Cref{prop reduced}, then we may assume that $G$ is the directed graph with two vertices $\{v_\lambda,v_0\}$ and two edges $\{w_\lambda,e_\lambda\}$ such that
\begin{equation*}
  r(w_\lambda)=s(w_\lambda)=r(e_\lambda)=v_\lambda,\quad s(e_\lambda)=v_0.
\end{equation*}
Then $L_{s(e_\lambda)}\mathcal{H}_G=\mathbb{C}\xi_{v_0}$.
It follows that $\|A\|=\|f\|_{H^2(\mathbb{D})}$.
Therefore, $A$ is the norm limit of $\sum_{n=0}^{N}a_nL_{w_\lambda^ne_\lambda}$ as $N\to\infty$.
Hence $A\in\mathcal{A}_G$.
\end{proof}

Note that the free semigroupoid algebra $\mathfrak{L}_G$ and the tensor algebra $\mathcal{A}_G$ of a fruit tree $G$ is infinite dimensional if it has at least one fruit.

\begin{proposition}\label{prop inner-imply-no-in-fruit}
Let $G$ be a generalized fruit tree.
If every bounded derivation on $\mathcal{A}_G$ is inner, then $G$ contains no in-fruit.
\end{proposition}

\begin{proof}
Suppose $G=G_t\sqcup_VG_f$ is a generalized fruit tree, $\mathscr{C}_{n_\lambda}$ is an in-fruit, $e_\lambda$ is the edge with $r(e_\lambda)=v_\lambda$ and $s(e_\lambda)\in\mathcal{V}_t\backslash V$, and $w_\lambda$ is the unique minimal circle at $v_\lambda$.
Let $f\in H^\infty(\mathbb{D})\backslash A(\mathbb{D})$.
For any vertex $v\in\mathcal{V}$ and edge $e\in\mathcal{E}$, we define
\begin{equation*}
  \delta(L_v)=0,\quad
  \delta(L_e)=
  \begin{cases}
    f(L_{w_\lambda})L_{e_\lambda}, & \mbox{if } e=e_\lambda, \\
    0, & \mbox{otherwise}.
  \end{cases}
\end{equation*}
Let $\mathscr{A}=\mathrm{span}\{L_w\colon w\in\mathbb{F}_G^+\}$.
Then $\mathscr{A}$ is a dense subalgebra of $\mathcal{A}_G$ and $\delta$ can be naturally extended to a derivation on $\mathscr{A}$.

Let $\Lambda_1=\{p_1e_\lambda p_2\colon p_1,p_2\in\mathbb{F}_G^+\}$, $\Lambda_2=\mathbb{F}_G^+\backslash\Lambda_1$, and $P_j$ the projection from $\mathcal{H}_G$ onto $\{\xi_p\colon p\in\Lambda_j\}$ for $j=1,2$.
Note that every $A\in\mathscr{A}$ can be written as a finite sum
\begin{equation*}
  \sum_{p\in\Lambda_1}a_pL_p+\sum_{p\in\Lambda_2}a_pL_p.
\end{equation*}
Let $A_j=\sum_{p\in\Lambda_j}a_pL_p$ for $j=1,2$.
Then $A_1P_1=0$.
For any $x\in\mathcal{H}_G$, we have
\begin{equation*}
  \|A_1x\|^2=\|A_1P_2x\|^2\leqslant\|A_1P_2x\|^2+\|A_2P_2x\|^2
  =\|AP_2x\|^2\leqslant\|A\|^2\|x\|^2.
\end{equation*}
It follows that $\|A_1\|\leqslant\|A\|$.
For each $v\in\mathcal{V}(\mathscr{C}_{n_\lambda})$, let $p_v$ be the path from $v_\lambda$ to $v$ with minimal lengths.
Then $A_1$ can be written as
\begin{equation*}
  A_1=\sum_{v\in\mathcal{V}(\mathscr{C}_{n_\lambda}),r(p_2)=s(e_\lambda)}
  L_{p_v}g_{v,p_2}(L_{w_\lambda})L_{e_\lambda}L_{p_2}.
\end{equation*}
Let $A_v=\sum_{r(p_2)=s(e_\lambda)}g_{v,p_2}(L_{w_\lambda})L_{e_\lambda}L_{p_2}$.
Then $A_1=\sum_{v\in\mathcal{V}(\mathscr{C}_{n_\lambda})}L_{p_v}A_v$.
By the definition of $\delta$, we have $\delta(A_v)=f(L_{w_\lambda})A_v$.
It follows that
\begin{equation*}
  \delta(A)=\delta(A_1)
  =\sum_{v\in\mathcal{V}(\mathscr{C}_{n_\lambda})}L_{p_v}f(L_{w_\lambda})A_v.
\end{equation*}
Hence for any vector $x\in\mathcal{H}_G$, it follows from \Cref{lem fruit-tree-norm} (i) that
\begin{align*}
  \|\delta(A)x\|^2 &=\sum_{v\in\mathcal{V}(\mathscr{C}_{n_\lambda})}
  \|L_{p_v}f(L_{w_\lambda})A_vx\|^2\leqslant\|f\|_{H^\infty(\mathbb{D})}^2
  \sum_{v\in\mathcal{V}(\mathscr{C}_{n_\lambda})}\|A_vx\|^2\\
  &=\|f\|_{H^\infty(\mathbb{D})}^2\|A_1x\|^2
  \leqslant\|f\|_{H^\infty(\mathbb{D})}^2\|A\|^2\|x\|^2.
\end{align*}
Therefore, $\|\delta(A)\|\leqslant\|f\|_{H^\infty(\mathbb{D})}\|A\|$ for every $A\in\mathscr{A}$.
Thus, $\delta$ extends to a bounded derivation on $\mathcal{A}_G$ with $\|\delta\|\leqslant 1$.
We claim that $\delta$ is not inner.

Suppose $\delta=\delta_T$ for some $T\in\mathcal{A}_G$.
Then $\delta_T(L_v)=0$ for all $v\in\mathcal{V}$ and we can write
\begin{equation*}
  T=\sum_{v\in\mathcal{V}}L_vTL_v.
\end{equation*}
Let $L_{v_\lambda}TL_{v_\lambda}=g(L_{w_\lambda})$ and $L_{s(e_\lambda)}TL_{s(e_\lambda)}=\lambda L_{s(e_\lambda)}$.
Then
\begin{equation*}
  f(L_{w_\lambda})L_{e_\lambda}=\delta(L_{e_\lambda})
  =\lambda L_{e_\lambda}-g(L_{w_\lambda})L_{e_\lambda}.
\end{equation*}
It follows that $g=\lambda-f\in H^\infty(\mathbb{D})\backslash A(\mathbb{D})$.
Therefore, $L_{v_\lambda}TL_{v_\lambda}=g(L_{w_\lambda})\notin\mathcal{A}_G$ by \Cref{lem fruit-tree-norm} (i).
That is a contradiction.
\end{proof}

\begin{theorem}\label{thm fruit-tree}
Let $G$ be a fruit tree.
Then every bounded derivation from $\mathcal{A}_G$ into $\mathfrak{L}_G$ is inner.
\end{theorem}

\begin{proof}
Suppose $G=G_t\sqcup_VG_f$, where $V=\{v_j\}_{j=1}^m$, $G_f$ is a directed graph with connected components $\{\mathscr{C}_{n_j}\}_{j=1}^m$, and $w_j$ is the minimal circle at $v_j$.

Let $P_0=\sum_{v\in\mathcal{V}_t\backslash V}L_v$, $P_j=\sum_{v\in\mathcal{V}(\mathscr{C}_{n_j})}L_v$ for every $1\leqslant j\leqslant m$, and $\delta$ a bounded derivation from $\mathcal{A}_G$ into $\mathfrak{L}_G$.
By \Cref{lem idempotent}, we may assume that $\delta(P_j)=0$ for every $0\leqslant j\leqslant m$.
Then $\delta$ can be viewed as a derivation from $P_j\mathcal{A}_GP_j$ into $P_j\mathfrak{L}_GP_j$.
By \Cref{thm S-C} and \Cref{lem tree}, there exists $T_j\in P_j\mathfrak{L}_GP_j$ such that $\delta(A)=\delta_{T_j}(A)$ for every $A\in P_j\mathcal{A}_GP_j$.
Therefore, $\delta$ is locally inner with respect to $\mathscr{A}=\bigcup_{j=0}^{n}P_j\mathcal{A}_GP_j$.

Assume that $\delta|_{\mathscr{A}}=0$.
For any $1\leqslant j\leqslant m$, there exists an edge $e_j$ between $v_j$ and $x_j\in\mathcal{V}_t\backslash V$, where $x_j$ may be repeated.
Since $\delta(L_{e_j})\in L_{r(e_j)}\mathfrak{L}_GL_{s(e_j)}$, we can write
\begin{equation*}
  \delta(L_{e_j})=
  \begin{cases}
    L_{e_j}A_j, & \mbox{if } s(e_j)=v_j, \\
    -A_jL_{e_j}, & \mbox{if } r(e_j)=v_j,
  \end{cases}
\end{equation*}
where $A_j=f_j(L_{w_j})$ and $f_j\in H^2(\mathbb{D})$.
If $s(e_j)=v_j$, then $A_j=L_{e_j}^*\delta(L_{e_j})$ is bounded, and hence $A_j\in\mathfrak{L}_{\mathscr{C}_{n_j}}$.
If $r(e_j)=v_j$, then $\delta(L_{e_j})=-f_j(L_{w_j})L_{e_j}$.
By \Cref{lem fruit-tree-norm} (ii), for any polynomial $g\in H^2(\mathbb{D})$, we have
\begin{equation*}
  \|f_jg\|_{H^2(\mathbb{D})}=\|f_j(L_{w_j})g(L_{w_j})L_{e_j}\|
  =\|\delta(g(L_{w_j})L_{e_j})\|
  \leqslant\|\delta\|\|g\|_{H^2(\mathbb{D})}.
\end{equation*}
It follows that $f_j\in H^\infty(\mathbb{D})$.
Hence $A_j\in\mathfrak{L}_{\mathscr{C}_{n_j}}$ by \Cref{lem fruit-tree-norm} (i).

Let $B_j=n_j\alpha_j(A_j)\in Z(\mathfrak{L}_{\mathscr{C}_{n_j}})$, where $\alpha_j\in\mathscr{D}_{\mathscr{C}_{n_j}}$ is the mapping given by the proof of \Cref{thm weak-Dixmier}.
Let $B=\sum_{j=1}^mB_j$.
Then $\delta=\delta_B$.
This completes the proof.
\end{proof}

\begin{remark}\label{rem fruit-tree}
Let $G$ be a fruit tree that contains no in-fruit.
By \Cref{lem idempotent}, \Cref{prop circle-inner}, \Cref{lem tree}, \Cref{lem Lp*}, \Cref{thm weak-Dixmier}, and the proof of \Cref{thm fruit-tree}, every bounded derivation on $\mathcal{A}_G$ is inner.
This can be viewed as a converse of \Cref{prop inner-imply-no-in-fruit} for fruit trees.
\end{remark}

Next, we provide an abstract characterization of generalized fruit trees.
A polygon in $G$ is called a {\sl fake circle} if it is not an $n$-circle.

\begin{proposition}\label{prop abstract-fruit-tree}
A connected directed graph $G$ with $|\mathcal{V}(G)|\geqslant 2$ is a generalized fruit tree if and only if $G$ satisfies the following properties:
\begin{enumerate}[(i)]
\item every maximal strongly connected component of $G$ is either a vertex graph or an $n$-circle graph;

\item there is no edge between any two nontrivial maximal strongly connected components;

\item for any nontrivial maximal strongly connected component $G_0$ of $G$, there is exactly one edge between $\mathcal{V}(G_0)$ and $\mathcal{V}(G)\backslash\mathcal{V}(G_0)$;

\item $G$ contains no fake circle.
\end{enumerate}
\end{proposition}

\begin{proof}
It is clear that every generalized fruit tree satisfies the properties (i-iv).
Suppose that $G$ satisfies the properties (i-iv).
Let $G_f$ be the subgraph of all nontrivial maximal strongly connected components of $G$.
By (i) and (ii), $G_f$ has connected components $\{\mathscr{C}_{n_\lambda}\}_{\lambda\in\Lambda}$.
By (ii) and (iii), there exists exactly one edge $e_\lambda$ between $v_\lambda\in\mathcal{V}(\mathscr{C}_{n_\lambda})$ and $\mathcal{V}_t'=\mathcal{V}\backslash\mathcal{V}_f$.
Let $G_t'$ be the induced subgraph of $G$ by $\mathcal{V}_t'$.
Then $G_t'$ contains no circle.
Since $G$ is connected and contains no fake circle by (iv), $G_t'$ is connected and contains no polygon.
It follows that $G_t'$ is a generalized tree.
Let $G_t$ be the graph with $\mathcal{V}_t=\{v_\lambda\}_{\lambda\in\Lambda}\cup\mathcal{V}_t'$ and $\mathcal{E}_t=\{e_\lambda\}_{\lambda\in\Lambda}\cup\mathcal{E}_t'$.
Let $V=\{v_\lambda\}_{\lambda\in\Lambda}$.
It is clear that $G=G_t\sqcup_VG_f$.
This completes the proof.
\end{proof}

\subsection{A certain converse}\label{subsec converse}
In this subsection, we will prove a converse to \Cref{thm fruit-tree}.
An edge in a directed graph is called an {\sl acyclic edge} if there exists no circle containing it.
In other words, $e$ is an acyclic edge if and only if there is no path from $r(e)$ to $s(e)$.
Note that a connected directed graph contains an acyclic edge if and only it is not strongly connected.
This fact has been used in the proof of \Cref{prop center}.

\begin{lemma}\label{lem derivation-extension}
Suppose $G$ is a directed graph, $e_0$ is an acyclic edge, $v_1=r(e_0)$, and $v_2=s(e_0)$.
Let $w_j\in\mathbb{F}_G^+$ with $r(w_j)=s(w_j)=v_j$ for $j=1,2$.
For any vertex $v\in\mathcal{V}$ and edge $e\in\mathcal{E}$, we define
\begin{equation*}
  \delta(L_v)=0,\quad
  \delta(L_e)=
  \begin{cases}
    L_{w_1e_0w_2}, & \mbox{if } e=e_0, \\
    0, & \mbox{otherwise}.
  \end{cases}
\end{equation*}
Then $\delta$ can be extended to a bounded derivation on $\mathcal{A}_G$ with $\|\delta\|\leqslant 1$.
\end{lemma}

\begin{proof}
Let $\mathscr{A}=\mathrm{span}\{L_w\colon w\in\mathbb{F}_G^+\}$.
Similar to the proof of \Cref{prop inner-imply-no-in-fruit}, it suffices to shows that $\|\delta|_{\mathscr{A}}\|\leqslant 1$.

Let $\Lambda_1=\{p_1e_0p_2\colon p_1,p_2\in\mathbb{F}_G^+\}$ and $\Lambda_2=\mathbb{F}_G^+\backslash\Lambda_1$.
Then every $A\in\mathscr{A}$ can be written as a finite sum $\sum_{p\in\Lambda_1}a_pL_p+\sum_{p\in\Lambda_2}a_pL_p$.
Let $A_j=\sum_{p\in\Lambda_j}a_pL_p$ for $j=1,2$.
Since $e_0$ is an acyclic edge, we have $\|A_1\|\leqslant\|A\|$.
Note that $A_1$ can be written as $A_1=\sum_{p}L_pL_{e_0}A_p$.
By the definition of $\delta$, we have
\begin{equation*}
  \delta(A)=\delta(A_1)=\sum_{p}L_pL_{w_1e_0w_2}A_p.
\end{equation*}
For any path $p_1\ne p_2$, the ranges of $L_{p_1e_0}$ and $L_{p_2e_0}$ are orthogonal.
Hence for any vector $x\in\mathcal{H}_G$, we have
\begin{align*}
  \|\delta(A)x\|^2 &=\sum_{p}\|L_pL_{w_1e_0w_2}A_px\|^2
  =\sum_{p}\|A_px\|^2=\|A_1x\|^2\leqslant\|Ax\|^2.
\end{align*}
Therefore, $\|\delta(A)\|\leqslant\|A\|$ for every $A\in\mathscr{A}$.
The proof is completed.
\end{proof}

Let $G$ be a connected directed graph that is not strongly connected.
We will show that if every bounded derivation on $\mathcal{A}_G$ is of the form $\delta_T$ for some $T\in\mathfrak{L}_G$, then $G$ satisfies the properties in \Cref{prop abstract-fruit-tree}.

\begin{lemma}\label{lem inner-imply-circle}
Let $G$ be a connected directed graph that is not strongly connected.
If every bounded derivation on $\mathcal{A}_G$ is of the form $\delta_T$ for some $T\in\mathfrak{L}_G$, then every maximal strongly connected component of $G$ is either a vertex graph or an $n$-circle graph.
\end{lemma}

\begin{proof}
Suppose there exists a nontrivial maximal strongly connected component $G_0$ of $G$ which is not an $n$-circle.
Since $G$ is connected, there exists an edge $e_0\in\mathcal{E}$ between $\mathcal{V}_0:=\mathcal{V}(G_0)$ and $\mathcal{V}\backslash\mathcal{V}_0$.
Since $G_0$ is maximal, $e_0$ is an acyclic edge.
Let $v_1=r(e_0)$ and $v_2=s(e_0)$.

First, we assume that $v_1\in\mathcal{V}_0$ and $v_2\in\mathcal{V}\backslash\mathcal{V}_0$.
Let $w_1$ be a circle at $v_1$, $w_2=v_2$, and $\delta$ the derivation given by \Cref{lem derivation-extension}.
Suppose $\delta=\delta_T$ for some $T\in\mathfrak{L}_G$.
Then $\delta_T(L_v)=0$ for all $v\in\mathcal{V}$ and we can write
\begin{equation*}
  T=\sum_{v\in\mathcal{V}}L_vTL_v.
\end{equation*}
Since $\delta_T|_{L_{v_1}\mathcal{A}_{G_0}L_{v_1}}=0$ and $Z(L_{v_1}\mathcal{A}_{G_0}L_{v_1})=\mathbb{C}L_{v_1}$ by \Cref{prop reduced} and \Cref{prop center}, we have $L_{v_1}TL_{v_1}=\lambda L_{v_1}$ for some $\lambda\in\mathbb{C}$.
It follow that
\begin{equation*}
  L_{w_1e_0}=\delta(L_{e_0})=\delta_T(L_{e_0})
  =L_{e_0}(L_{v_2}TL_{v_2}-\lambda I).
\end{equation*}
This is a contradiction.

Next, we assume that $v_1\in\mathcal{V}\backslash\mathcal{V}_0$ and $v_2\in\mathcal{V}_0$.
Let $w_2$ be a circle at $v_2$, $w_1=v_1$, and $\delta$ the derivation given by \Cref{lem derivation-extension}.
Similarly, we can get
\begin{equation*}
  L_{e_0w_2}=(\lambda I-L_{v_1}TL_{v_1})L_{e_0}.
\end{equation*}
That is a contradiction.
We complete the proof.
\end{proof}

\begin{lemma}\label{lem inner-imply-no-edge}
Let $G$ be a connected directed graph that is not strongly connected.
If every bounded derivation on $\mathcal{A}_G$ is of the form $\delta_T$ for some $T\in\mathfrak{L}_G$, then there is no edge between any two nontrivial maximal strongly connected components.
\end{lemma}

\begin{proof}
Suppose $G_1$ and $G_2$ are two nontrivial maximal strongly connected components and $e_0$ is an edge between them.
Then $e_0$ is an acyclic edge.
Let $v_1=r(e_0)$ and $v_2=s(e_0)$.

Without loss of generality, we assume that $v_j\in\mathcal{V}_j=\mathcal{V}(G_j)$ for $j=1,2$.
Let $w_j$ be a circle at $v_j$ for $j=1,2$, and $\delta$ the derivation given by \Cref{lem derivation-extension}.
Suppose $\delta=\delta_T$ for some $T\in\mathfrak{L}_G$.
Then $L_{w_1e_0w_2}=L_{e_0}T-TL_{e_0}$.
That is a contradiction.
\end{proof}

\begin{lemma}\label{lem inner-imply-one-edge}
Let $G$ be a connected directed graph that is not strongly connected.
If every bounded derivation on $\mathcal{A}_G$ is of the form $\delta_T$ for some $T\in\mathfrak{L}_G$, then for any nontrivial maximal strongly connected component $G_0$ of $G$, there is exactly one edge between $\mathcal{V}(G_0)$ and $\mathcal{V}(G)\backslash\mathcal{V}(G_0)$.
\end{lemma}

\begin{proof}
Suppose there are two distinct edges $e_1$ and $e_2$ between $\mathcal{V}_0=\mathcal{V}(G_0)$ and $\mathcal{V}\backslash\mathcal{V}_0$.
Let $v_1=r(e_1)$, $v_2=s(e_1)$, $v_3=r(e_2)$, and $v_4=s(e_2)$.

First, we assume that $v_1,v_3\in\mathcal{V}_0$ and $v_2,v_4\in\mathcal{V}\backslash\mathcal{V}_0$.
Note that $e_1$ is an acyclic edge.
Let $w_1$ be a circle at $v_1$, $w_2=v_2$, and $\delta$ the derivation given by \Cref{lem derivation-extension}.
Suppose $\delta=\delta_T$ for some $T\in\mathfrak{L}_G$.
Then
\begin{equation*}
  T=\sum_{v\in\mathcal{V}}L_vTL_v.
\end{equation*}
Since $\delta(L_{e_1})=L_{w_1e_1}$, we have
\begin{equation*}
  L_{e_1}\cdot L_{v_2}TL_{v_2}=(L_{v_1}TL_{v_1}+L_{w_1})L_{e_1}.
\end{equation*}
It follows that $L_{v_1}TL_{v_1}=\lambda_1L_{v_1}-L_{w_1}$ and $L_{v_2}TL_{v_2}=\lambda_1L_{v_2}$.
Since $\delta(L_{e_2})=0$ and $e_2$ is an acyclic edge, we have $L_{v_3}TL_{v_3}=\lambda_2L_{v_3}$ and $L_{v_4}TL_{v_4}=\lambda_2L_{v_4}$.
Let $p$ be a path of $G_0$ from $v_3$ to $v_1$.
Then $\delta(L_p)=0$ implies that $L_{w_1}=(\lambda_1-\lambda_2)L_{v_1}$.
That is a contradiction.
Similar to the proof of \Cref{lem inner-imply-circle}, we can get a contradiction for other cases.
This completes the proof.
\end{proof}

\begin{lemma}\label{lem inner-imply-no-fake}
Let $G$ be a connected directed graph that is not strongly connected.
If every bounded derivation on $\mathcal{A}_G$ is of the form $\delta_T$ for some $T\in\mathfrak{L}_G$, then $G$ contains no fake circle.
\end{lemma}

\begin{proof}
Suppose there exists a fake circle, i.e., a finite set $\mathscr{F}=\{v_j,e_j\}_{j=1}^{n}$ such that $\{v_j\}_{j=1}^{n}$ is a set of distinct vertices and $\{e_j\}_{j=1}^{n}$ is a set of edges such that $\{r(e_j),s(e_j)\}=\{v_j,v_{j+1\,(\mathrm{mod}\,n)}\}$ for each $1\leqslant j\leqslant n$.
For convenience, let $v_{n+1}=v_1$.
If $L_{v_j}\mathcal{A}_GL_{v_{j+1}}\ne\{0\}$ and $L_{v_{j+1}}\mathcal{A}_GL_{v_j}\ne\{0\}$ for every $1\leqslant j\leqslant n$, then $\mathscr{F}$ is contained in a maximal strongly connected component of $G$.
By \Cref{lem inner-imply-circle}, $\mathscr{F}$ must be a circle.
That is a contradiction.

Without loss of generality, we assume that $L_{v_2}\mathcal{A}_GL_{v_1}=0$.
Then $e_1$ is an acyclic edge such that $r(e_1)=v_1$ and $s(e_1)=v_2$.
Let $w_1=v_1$, $w_2=v_2$, and $\delta$ the derivation given by \Cref{lem derivation-extension}.
Suppose $\delta=\delta_T$ for some $T\in\mathfrak{L}_G$.
Similar to the proof of \Cref{lem inner-imply-circle}, we have
\begin{equation*}
  T=\sum_{v\in\mathcal{V}}L_vTL_v.
\end{equation*}
Since $\delta(L_{e_1})=L_{e_1}$, we have
\begin{equation*}
  L_{e_1}(L_{v_2}TL_{v_2}-I)=L_{v_1}TL_{v_1}\cdot L_{e_1}.
\end{equation*}
It follows that $L_{v_1}TL_{v_1}=\lambda_1L_{v_1}$, $L_{v_2}TL_{v_2}=\lambda_2L_{v_2}$, and $\lambda_2=1+\lambda_1$.
Note that
\begin{equation*}
  0=\delta(L_{e_2})=
  \begin{cases}
   L_{e_2}\cdot L_{v_3}TL_{v_3}- \lambda_2L_{e_2}, & \mbox{if } r(e_2)=v_2, \\
   \lambda_2L_{e_2}- L_{v_3}TL_{v_3}\cdot L_{e_2}, & \mbox{otherwise}.
  \end{cases}
\end{equation*}
In both cases, we have $L_{v_3}TL_{v_3}=\lambda_2L_{v_3}$.
Inductively, since $\delta(L_{e_j})=0$, we have $L_{v_{j+1}}TL_{v_{j+1}}=\lambda_2L_{v_{j+1}}$ for every $2\leqslant j\leqslant n$.
It follows that $L_{v_1}TL_{v_1}=\lambda_2L_{v_1}$.
That is a contradiction.
\end{proof}

Now we are ready to prove a certain converse to \Cref{thm fruit-tree}.

\begin{theorem}\label{thm converse}
Let $G$ be a connected directed graph that is not strongly connected.
If every bounded derivation on $\mathcal{A}_G$ is of the form $\delta_T$ for some $T\in\mathfrak{L}_G$, then $G$ is a generalized fruit tree.
\end{theorem}

\begin{proof}
Since $G$ is not strongly connected, we have $|\mathcal{V}|\geqslant 2$.
By \Cref{lem inner-imply-circle}, \Cref{lem inner-imply-no-edge}, \Cref{lem inner-imply-one-edge}, and \Cref{lem inner-imply-no-fake}, $G$ satisfies the properties (i-iv) in \Cref{prop abstract-fruit-tree}.
Therefore, $G$ is a generalized fruit tree.
\end{proof}

By \Cref{lem infinite-not-inner}, \Cref{prop inner-imply-no-in-fruit}, \Cref{thm converse}, and \Cref{rem fruit-tree}, we have the following result.

\begin{corollary}\label{cor converse}
Let $G$ be a connected directed graph that is not strongly connected.
Then every bounded derivation on $\mathcal{A}_G$ is inner if and only if $G$ is a fruit tree that contains no in-fruit.
\end{corollary}

\subsection{Alternating number}\label{subsec alternating}
Let $G$ be a connected directed graph that is not strongly connected.
If every bounded derivation from $\mathcal{A}_G$ into $\mathfrak{L}_G$ is inner, then $G$ must be a generalized fruit tree by \Cref{thm converse}.
Conversely, if $G$ is a fruit tree, then every bounded derivation from $\mathcal{A}_G$ into $\mathfrak{L}_G$ is inner by \Cref{thm fruit-tree}.
For a generalized fruit tree $G$ with infinitely many vertices, a natural question is whether every bounded derivation from $\mathcal{A}_G$ into $\mathfrak{L}_G$ is inner.
For this propose, we introduce the concept of alternating numbers in this subsection.

Suppose $G$ is a directed graph.
A {\sl line} $L$ of length $n\geqslant 1$ in $G$ is a finite set $\{v_j\}_{j=1}^{n+1}\cup\{e_j\}_{j=1}^{n}$ such that $\{v_j\}_{j=1}^{n+1}$ is a set of distinct vertices and $\{e_j\}_{j=1}^{n}$ is a set of edges such that
\begin{equation*}
  \{r(e_j),s(e_j)\}=\{v_j,v_{j+1}\}
\end{equation*}
for each $1\leqslant j\leqslant n$.
Note that if we require $v_{n+1}=v_1$, then the line $L$ becomes a polygon.

\begin{definition}\label{def alternating}
Suppose $G$ is a directed graph and $L$ is a line in $G$.
A vertex $v$ in $L$ is called an {\sl alternating vertex} of $L$ if either $\deg_L^+(v)=2$ or $\deg_L^-(v)=2$.
We denote by $A(L)$ the number of alternating vertices in $L$.
The {\sl alternating number} of $G$ is defined as $A(G)=\sup\{A(L)\colon L~\text{is a line in}~G\}$.
\end{definition}

If $G=G_t\sqcup_VG_f$ is a generalized fruit tree, then it is easy to see that
\begin{equation*}
  A(G_t)\leqslant A(G)\leqslant A(G_t)+2.
\end{equation*}

\begin{lemma}\label{lem alternating}
Suppose $G=G_t\sqcup_VG_f$ is a generalized fruit tree such that $A(G_t)\geqslant 2m$.
Then there exists a nonzero bounded derivation $\delta$ on $\mathcal{A}_G$ such that if $\delta=\delta_T$ for some $T\in\mathfrak{L}_G$, then $\|T\|\geqslant\frac{m}{2}\|\delta\|$.
\end{lemma}

\begin{proof}
Suppose $L=\{v_j\}_{j=1}^{n+1}\cup\{e_j\}_{j=1}^{n}$ is a line in $G_t$ such that $A(L)=2m$.
Let $\{v_{k_j}\}_{j=1}^{2m}$ be the $2m$ alternating vertices of $L$ such that $\{k_j\}_{j=1}^{2m}$ is increasing.
Without loss of generality, we assume that
\begin{equation*}
  \deg^+(v_{k_{2j-1}})=\deg^-(v_{k_{2j}})=2
\end{equation*}
for every $1\leqslant j\leqslant m$.
For any vertex $v\in\mathcal{V}$ and edge $e\in\mathcal{E}$, we define
\begin{equation*}
  \delta(L_v)=0,\quad
  \delta(L_e)=
  \begin{cases}
    L_{e_{k_{2j-1}}}, & \mbox{if } e=e_{k_{2j-1}}, 1\leqslant j\leqslant m, \\
    0, & \mbox{otherwise}.
  \end{cases}
\end{equation*}
Let $\mathscr{A}=\mathrm{span}\{L_w\colon w\in\mathbb{F}_G^+\}$.
Then $\mathscr{A}$ is a subalgebra of $\mathcal{A}_G$ and $\delta$ can be naturally extended to a derivation on $\mathscr{A}$.
Similar to the proof of \Cref{lem derivation-extension}, we will show that $\|\delta|_{\mathscr{A}}\|\leqslant 1$.

Let $\Lambda_1=\{p_1e_{k_{2j-1}}p_2\colon p_1,p_2\in\mathbb{F}_G^+,1\leqslant j\leqslant m\}$, $\Lambda_2=\mathbb{F}_G^+\backslash\Lambda_1$, and $P_j$ the projection from $\mathcal{H}_G$ onto $\{\xi_p\colon p\in\Lambda_j\}$ for $j=1,2$.
Note that every $A\in\mathscr{A}$ can be written as a finite sum
\begin{equation*}
  \sum_{p\in\Lambda_1}a_pL_p+\sum_{p\in\Lambda_2}a_pL_p.
\end{equation*}
Let $A_j=\sum_{p\in\Lambda_j}a_pL_p$ for $j=1,2$.
Since every path $p\in\mathbb{F}_G^+$ contains at most one $e_{k_{2j-1}}$, we have $\delta(A)=A_1$.
Note that $A_1P_1=0$.
For any $x\in\mathcal{H}_G$, we have
\begin{equation*}
  \|A_1x\|^2=\|A_1P_2x\|^2\leqslant\|A_1P_2x\|^2+\|A_2P_2x\|^2
  =\|AP_2x\|^2\leqslant\|A\|^2\|x\|^2.
\end{equation*}
It follows that $\|\delta(A)\|=\|A_1\|\leqslant\|A\|$ for every $A\in\mathscr{A}$.
Therefore, $\delta$ extends to a derivation on $\mathcal{A}_G$ with $\|\delta\|\leqslant 1$.

Suppose $\delta=\delta_T$ for some $T\in\mathfrak{L}_G$.
Then we can write
\begin{equation*}
  T=\sum_{v\in\mathcal{V}}L_vTL_v.
\end{equation*}
Since $L_{v_{k_j}}\mathfrak{L}_GL_{v_{k_j}}=\mathbb{C}L_{v_{k_j}}$, there exists $\lambda_j\in\mathbb{C}$ such that $L_{v_{k_j}}TL_{v_{k_j}}=\lambda_jL_{v_{k_j}}$ for every $1\leqslant j\leqslant 2m$.
Let
\begin{equation*}
  p_i=
  \begin{cases}
    e_{k_{2j}-1}e_{k_{2j}-2}\cdots e_{k_{2j-1}}, & \mbox{if } i=2j-1~\text{and}~1\leqslant j\leqslant m,\\
    e_{k_{2j}}e_{k_{2j}+1}\cdots e_{k_{2j+1}-1}, & \mbox{if } i=2j~\text{and}~1\leqslant j\leqslant m-1.
  \end{cases}
\end{equation*}
Then $\delta(L_{p_{2j-1}})=L_{p_{2j-1}}$ for every $1\leqslant j\leqslant m$ and $\delta(L_{p_{2j}})=0$ for every $1\leqslant j\leqslant m-1$.
It follows that $\lambda_{2j-1}=1+\lambda_{2j}$ for every $1\leqslant j\leqslant m$ and $\lambda_{2j}=\lambda_{2j+1}$ for every $1\leqslant j\leqslant m-1$.
Therefore, $\lambda_{1}-\lambda_{2m}=m$.
In particular, we have
\begin{equation*}
  \|T\|\geqslant\max\{|\lambda_1|,|\lambda_{2m}|\}\geqslant\frac{m}{2}.
\end{equation*}
This completes the proof.
\end{proof}

\begin{proposition}\label{prop alternating}
Let $G$ be a generalized fruit tree.
If every bounded derivation on $\mathcal{A}_G$ is of the form $\delta_T$ for some $T\in\mathfrak{L}_G$, then $A(G)<\infty$.
\end{proposition}

\begin{proof}
Let $\mathcal{D}
=\{T\in\mathfrak{L}_G\colon\delta_T(\mathcal{A}_G)\subseteq\mathcal{A}_G\}$.
Then $\mathcal{D}$ is a Banach subspace of $\mathfrak{L}_G$.
We define a bounded linear map
\begin{equation*}
  \varphi\colon\mathcal{D}/Z(\mathfrak{L}_G)\to\mathrm{Der}(\mathcal{A}_G),\quad
  T+Z(\mathfrak{L}_G)\mapsto\delta_T.
\end{equation*}
Since $\varphi$ is bijective by assumption, $\varphi^{-1}$ is also bounded by the open mapping theorem.
Suppose on the contrary that $A(G)=\infty$.
By \Cref{lem alternating}, for every $m\geqslant 1$, there exists a nonzero bounded derivation $\delta_m$ on $\mathcal{A}_G$ such that if $\delta_m=\delta_{T_m}$ for some $T_m\in\mathfrak{L}_G$, then $\|T_m\|\geqslant\frac{m}{2}\|\delta_m\|$.
It follows that
\begin{equation*}
  \frac{m}{2}\|\delta_m\|\leqslant\|\varphi^{-1}(\delta_m)\|
  \leqslant\|\varphi^{-1}\|\|\delta_m\|.
\end{equation*}
That is a contradiction.
\end{proof}

For example, let $G$ be the directed graph with vertices $\{v_j\}_{j\in\mathbb{Z}}$ and edges $\{e_j\}_{j\in\mathbb{Z}}$ such that for all $j\in\mathbb{Z}$, we have $r(e_{2j-1})=v_{2j-1}$, $s(e_{2j-1})=v_{2j}$, $r(e_{2j})=v_{2j+1}$, and $s(e_{2j})=v_{2j}$.
See the following graph
\begin{equation*}
\xymatrix{
\cdots&v_0&v_1&v_2&v_3&v_4&\cdots.
\ar_{e_{-1}}"1,2";"1,1"
\ar^{e_0}"1,2";"1,3"
\ar_{e_1}"1,4";"1,3"
\ar^{e_2}"1,4";"1,5"
\ar_{e_3}"1,6";"1,5"
\ar^{e_4}"1,6";"1,7"
}
\end{equation*}
Then $G$ is a generalized tree with $A(G)=\infty$.
In particular, not every derivation on $\mathcal{A}_G$ is inner.

\begin{corollary}\label{cor inner-imply}
Suppose $G$ is a directed graph with connected components $\{G_{\lambda_1}\}_{\lambda_1\in\Lambda_1}
\cup\{G_{\lambda_2}\}_{\lambda_2\in\Lambda_2}$ such that each $G_{\lambda_1}$ is strongly connected and each $G_{\lambda_2}$ is not strongly connected.
If every bounded derivation on $\mathcal{A}_G$ is of the form $\delta_T$ for some $T\in\mathfrak{L}_G$, then each $G_{\lambda_2}$ is a generalized fruit tree and $\sup_{\lambda_2\in\Lambda_2}A(G_{\lambda_2})<\infty$.
\end{corollary}

\begin{proof}
It follows from \Cref{prop derivation-decomposition}, \Cref{thm converse}, \Cref{lem alternating}, and the proof of \Cref{prop alternating}.
\end{proof}

\subsection{A reduction}
Let $G=G_t\sqcup_VG_f$ be a generalized fruit tree.
By \Cref{def fruit-tree}, the set $\mathcal{V}_t'=\mathcal{V}\backslash\mathcal{V}_f=\mathcal{V}_t\backslash V$ is nonempty.
Let $G_t'$ be the reduced subgraph of $G$ by $\mathcal{V}_t'$, which is the same as the induced subgraph of $G$ by $\mathcal{V}_t'$.
It is clear that $G_t'$ is a generalized subtree of $G_t$.
The following proposition is an enhanced version of \Cref{thm fruit-tree}.
Combining with \Cref{lem reduced-derivation}, we only need to focus on generalized trees.
Recall that $\mathcal{A}_G=\mathfrak{L}_G$ if $G$ is a tree.

\begin{proposition}\label{prop reduction}
Suppose $G=G_t\sqcup_VG_f$ is a fruit tree, $G_t'$ is the reduced subgraph of $G$ by $\mathcal{V}_t'=\mathcal{V}\backslash\mathcal{V}_f$, and $C\geqslant 1$ is a constant.
Assume that for every derivation $\delta$ on $\mathcal{A}_{G_t'}$, there exists $T\in\mathcal{A}_{G_t'}$ such that $\delta=\delta_T$ and $\|T\|\leqslant C\|\delta\|$.
Then for every bounded derivation $\delta$ from $\mathcal{A}_G$ into $\mathfrak{L}_G$, there exists $T\in\mathfrak{L}_G$ such that $\delta=\delta_T$ and $\|T\|\leqslant 22C\|\delta\|$.
\end{proposition}

\begin{proof}
Suppose $V=\{v_j\}_{j=1}^m$ and $G_f$ has connected components $\{\mathscr{C}_{n_j}\}_{j=1}^m$.
For any $1\leqslant j\leqslant m$, there exists an edge $e_j$ between $v_j$ and $x_j\in\mathcal{V}_t'$, where $x_j$ may be repeated.
Without loss of generality, we assume that $\mathscr{C}_{n_j}$ is an out-fruit for each $1\leqslant j\leqslant m'$ and it is an in-fruit for each $m'+1\leqslant j\leqslant m$.
Let
\begin{align*}
   P_j&=\sum_{v\in\mathcal{V}(\mathscr{C}_{n_j})}L_v\quad\text{for every}~
   1\leqslant j\leqslant m,\\
   Q_0&=\sum_{v\in \mathcal{V}_t'}L_v,\quad Q_1=\sum_{j=1}^{m'}P_j,\quad Q_2=\sum_{j=m'+1}^mP_j.
\end{align*}
Clearly, we have $Q_0+Q_1+Q_2=I$.
Let $\delta$ be a bounded derivation from $\mathcal{A}_G$ into $\mathfrak{L}_G$.
We define
\begin{equation*}
  T_0=\sum_{j=0}^2Q_j\delta(Q_j)=(Q_1-Q_0)\delta(Q_1)+(Q_2-Q_0)\delta(Q_2).
\end{equation*}
Then $\|T_0\|\leqslant 2\|\delta\|$ and $\delta(Q_j)=\delta_{T_0}(Q_j)$ for $0\leqslant j\leqslant 2$ by \Cref{lem idempotent}.
Let $\delta_1=\delta-\delta_{T_0}$.
Then $\|\delta_1\|\leqslant 5\|\delta\|$ and $\delta_1(Q_j)=0$.
Moreover, for any $1\leqslant j\leqslant m'$, we have $\delta_1(P_j)=Q_1\delta_1(P_j)Q_1$.
Combining with $\delta_1(P_j)=\delta_1(P_j)P_j+P_j\delta_1(P_j)$, we can get $\delta_1(P_j)=0$ for every $1\leqslant j\leqslant m'$.
Similarly, $\delta_1(P_j)=0$ for every $m'+1\leqslant j\leqslant m$.

Since $\delta_1|_{Q_0\mathcal{A}_GQ_0}$ is a derivation on $Q_0\mathcal{A}_GQ_0\cong\mathcal{A}_{G_t'}$ by \Cref{prop reduced}, there exists $S_0\in Q_0\mathcal{A}_GQ_0$ such that $\delta_1|_{Q_0\mathcal{A}_GQ_0}=\delta_{S_0}$ and $\|S_0\|\leqslant C\|\delta_1\|$ by assumption.
Since $\delta_1|_{P_j\mathcal{A}_GP_j}$ is a derivation from $\mathcal{A}_{\mathscr{C}_j}$ into $\mathfrak{L}_{\mathscr{C}_j}$, by applying \Cref{thm S-C}, there exists $S_j\in\mathfrak{L}_{\mathscr{C}_j}$ such that $\delta_1|_{P_j\mathcal{A}_GP_j}=\delta_{S_j}$ and $\|S_j\|\leqslant\|\delta_1\|$.
Let $S=\sum_{j=0}^{m}S_j$ and $\delta_2=\delta_1-\delta_S$.
Then $\|S\|\leqslant C\|\delta_1\|$, $\|\delta_2\|\leqslant(1+2C)\|\delta_1\|$, and $\delta_2$ vanishes on $Q_0\mathcal{A}_GQ_0$ and $P_j\mathcal{A}_GP_j$ for every $1\leqslant j\leqslant m$.

Similar to the proof of \Cref{thm fruit-tree}, there exists $A_j\in\mathfrak{L}_{\mathscr{C}_{n_j}}$ with $\|A_j\|\leqslant\|\delta_2\|$ such that
\begin{equation*}
  \delta_2(L_{e_j})=
  \begin{cases}
    L_{e_j}A_j, & \mbox{if } 1\leqslant j\leqslant m', \\
    -A_jL_{e_j}, & \mbox{if } m'+1\leqslant j\leqslant m.
  \end{cases}
\end{equation*}
Let $B_j=n_j\alpha_j(A_j)\in Z(\mathfrak{L}_{\mathscr{C}_{n_j}})$, where $\alpha_j\in\mathscr{D}_{\mathscr{C}_{n_j}}$ is the mapping given by the proof of \Cref{thm weak-Dixmier}.
More precisely, let $\mathcal{V}(\mathscr{C}_{n_j})=\{v_{j1},v_{j2},\ldots,v_{jn_j}\}$ and $c_{ji}$ the minimal circle at $v_{ji}$ for every $1\leqslant i\leqslant n_j$, see \eqref{equ c-j}.
We may assume that $v_j=v_{j1}$ and $A_j=f_j(L_{c_{j1}})$, where $f_j\in H^{\infty}(\mathbb{D})$ by the proof of \Cref{thm fruit-tree}.
Then $B_j=\sum_{i=1}^{n_j}f_j(c_{ji})$ and $\|B_j\|=\|A_j\|$ since $f(c_{ji})f(c_{jk})=0$ for any $1\leq i\neq k\leq n_j$.
Let $B=\sum_{j=1}^mB_j$.
It is clear that $\delta_2=\delta_B$ and
\begin{equation*}
  \|B\|=\max_{1\leqslant j\leqslant m}\|B_j\|=\max_{1\leqslant j\leqslant m}\|A_j\|\leqslant\|\delta_2\|.
\end{equation*}
Let $T=T_0+S+B$.
Then $\delta=\delta_T$ and $\|T\|\leqslant(7+15C)\|\delta\|$.
This completes the proof.
\end{proof}

\begin{corollary}\label{cor reduction}
Suppose $G=G_t\sqcup_VG_f$ is a generalized fruit tree, and $C\geqslant 1$ is a constant.
Assume that for every subtree $G_0$ of $G_t$ and every derivation $\delta$ on $\mathcal{A}_{G_0}$, there exists $T\in\mathcal{A}_{G_0}$ such that $\delta=\delta_T$ and $\|T\|\leqslant C\|\delta\|$.
Then for every derivation $\delta$ from $\mathcal{A}_G$ into $\mathfrak{L}_G$, there exists $T\in\mathfrak{L}_G$ such that $\delta=\delta_T$ and $\|T\|\leqslant 22C\|\delta\|$.
\end{corollary}

\begin{proof}
Suppose $G_f$ has connected components $\{\mathscr{C}_{n_\lambda}\}_{\lambda\in\Lambda}$.
Let $F$ be a finite subset of $\mathcal{V}$ such that either $\mathcal{V}(\mathscr{C}_{n_\lambda})\cap F=\emptyset$ or $\mathcal{V}(\mathscr{C}_{n_\lambda})\subseteq F$ for each $\lambda\in\Lambda$, and the reduced subgraph of $G$ by $F$ is connected.
Combining with \Cref{prop reduction}, the rest part is the same as the proof of \Cref{lem reduced-derivation}.
\end{proof}

\subsection{Rooted trees}
Let $\{E_{ij}\}_{i,j=1}^{n}$ be a system of matrix units in the matrix algebra $M_n(\mathbb{C})$.
The {\sl upper triangle algebra} of order $n$ is defined as
\begin{equation*}
  UT_n=\mathrm{span}\{E_{ij}\colon 1\leqslant i\leqslant j\leqslant n\}.
\end{equation*}
Suppose $G$ is a directed graph.
We say that $G$ is the {\sl upper triangle tree} of order $n$ if it has vertices $\{v_j\}_{j=1}^n$ and edges $\{e_j\}_{j=1}^{n-1}$ such that $r(e_j)=v_j$ and $s(e_j)=v_{j+1}$.
Actually, let $e_{ij}$ be the unique path in $\mathbb{F}_G^+$ from $v_j$ to $v_i$, i.e., $e_{jj}=v_j$ for $1\leqslant j\leqslant n$ and $e_{ij}=e_ie_{i+1}\cdots e_{j-1}$ for $1\leqslant i<j\leqslant n$.
Then the linear map
\begin{equation*}
  \varphi\colon\mathcal{A}_G\to UT_n,\quad \sum_{i\leqslant j}a_{ij}e_{ij}\mapsto\sum_{i\leqslant j}a_{ij}E_{ij}
\end{equation*}
is an isomorphism.
Note that a tree $G$ is an upper triangle tree if and if only its alternating
number $A(G)$ is zero.

\begin{definition}
Let $G$ be a tree.
Then $G$ is called an {\sl out-tree} if there exists a vertex $v_s$ such that for each $v\in\mathcal{V}$, there exists a path from $v_s$ to $v$.
Note that $v_s$ is unique and is called the {\sl source root} of $G$.

Similarly, $G$ is called an {\sl in-tree} if there exists a vertex $v_r$ such that for each $v\in\mathcal{V}$, there exists a path from $v$ to $v_r$.
Note that $v_r$ is unique and is called the {\sl range root} of $G$.

We say that $G$ is a {\sl rooted tree} if it is either an out-tree or an in-tree.
\end{definition}

By definition, every upper triangle tree is an out-tree and an in-tree at the same time.
Let $G$ be an out-tree with $n$ vertices.
Let $v_n$ be the source root.
Then we can write $\mathcal{V}=\{v_1,v_1,\ldots,v_n\}$ such that there exists an edge between $v_j$ and $\{v_{j+1},v_{j+2},\ldots,v_n\}$ for every $1\leqslant j\leqslant n-1$.
It is clear that there exists no path from $v_i$ to $v_j$ for every $1\leqslant i<j\leqslant n$.

\begin{lemma}\label{lem out-tree}
Suppose $G$ is an out-tree with $n$ vertices.
Then $\mathcal{A}_G$ on $\mathcal{H}_G$ is completely isometrically isomorphic to a subalgebra of the upper triangle algebra $UT_n$ on $\mathbb{C}^n$.
\end{lemma}

\begin{proof}
Assume that $\mathcal{V}=\{v_1,v_2,\ldots,v_n\}$ such that $v_n$ is the source root and there exists no path from $v_i$ to $v_j$ for every $1\leqslant i<j\leqslant n$.
Let $\Lambda$ be the set of all pairs $(i,j)$ such that $1\leqslant i\leqslant j\leqslant n$ and there is a path $e_{ij}$ from $v_j$ to $v_i$.
We define a homomorphism by
\begin{equation*}
  \varphi\colon\mathcal{A}_G\to UT_n,\quad \sum_{(i,j)\in\Lambda}a_{ij}e_{ij}\mapsto\sum_{(i,j)\in\Lambda}a_{ij}E_{ij}.
\end{equation*}
Note that every vector $x\in\mathcal{H}_G$ is of the form $\sum_{j=1}^{n}R_{v_j}x$, and each vector $R_{v_j}x$ can be viewed as a vector $x_j$ in $\mathbb{C}^n$.
Moreover, since $v_n$ is the source root, the vector $x_n$ is taken over all vectors in $\mathbb{C}^n$.
Similar to the proof of \Cref{prop reduced}, $\varphi$ is a complete
isometric isomorphism from $\mathcal{A}_G$ into $UT_n$.
\end{proof}

\begin{remark}\label{rem out-tree}
By \Cref{lem out-tree}, $\mathcal{A}_G$ can be identified with a subalgebra of $UT_n$.
Since $G$ is an out-tree,
$\begin{pmatrix}
  0 & y
\end{pmatrix}$ lies in $\mathcal{A}_G$ for any column vector $y\in\mathbb{C}^n$.
\end{remark}

\begin{proposition}\label{prop out-tree}
Suppose $G$ is an out-tree.
Then for every derivation $\delta$ on $\mathcal{A}_G$, there exists $T\in\mathcal{A}_G$ such that $\delta=\delta_T$ and $\|T\|\leqslant\|\delta\|$.
\end{proposition}

\begin{proof}
By \Cref{lem out-tree}, we may assume that $\mathcal{A}_G$ is a subalgebra of $UT_n$.
By \Cref{lem tree}, there exists $T\in\mathcal{A}_G$ such that $\delta=\delta_T$.
Subtracting $\lambda I$ from $T$ for some scalar $\lambda$, we may assume that $T=
\begin{pmatrix}
  -S\\ 0
\end{pmatrix}$, where $S\in M_{n-1,n}(\mathbb{C})$.
For any vector $y\in\mathbb{C}^n$, we have
\begin{equation*}
  \delta(
  \begin{pmatrix}
    0 & y
  \end{pmatrix})=
  \begin{pmatrix}
    0 & Sy \\
    0 & 0
  \end{pmatrix}.
\end{equation*}
It follows that
\begin{equation*}
  \|Sy\|_2=\left\|
  \begin{pmatrix}
    0 & Sy \\
    0 & 0
  \end{pmatrix}\right\|
  \leqslant\|\delta\|\left\|\begin{pmatrix}
    0 & y
  \end{pmatrix}\right\|
  =\|\delta\|\|y\|_2.
\end{equation*}
Therefore, $\|T\|=\|S\|\leqslant\|\delta\|$.
\end{proof}

We consider an example of in-trees in the next proposition.

\begin{proposition}\label{prop in-tree-eg}
Suppose $G$ is the in-tree with vertices $\mathcal{V}=\{v_j\}_{j=0}^n$ and edges $\mathcal{E}=\{e_j\}_{j=1}^n$ such that $r(e_j)=v_0$ and $s(e_j)=v_j$ for $1\leqslant j\leqslant n$.
Then for every derivation $\delta$ on $\mathcal{A}_G$, there exists $T\in\mathcal{A}_G$ such that $\delta=\delta_T$ and $\|T\|\leqslant\|\delta\|$.
\end{proposition}

\begin{proof}
Let $\delta$ be a derivation on $\mathcal{A}_G$.
By \Cref{lem tree}, there exists $T\in\mathcal{A}_G$ such that $\delta=\delta_T$.
Subtracting $\lambda I$ from $T$ for some scalar $\lambda$, we may assume that
\begin{equation*}
  T=\sum_{j=1}^{n}s_jL_{v_j}+\sum_{j=1}^{n}t_jL_{e_j}.
\end{equation*}
For any vector $x=\sum_{j=0}^{n}x_j\xi_{v_j}+\sum_{j=1}^{n}y_j\xi_{e_j}$, we have
\begin{equation*}
  Tx=\sum_{j=1}^{n}s_jx_j\xi_{v_j}+\sum_{j=1}^{n}t_jx_j\xi_{e_j}.
\end{equation*}
It follows that $\|Tx\|_2^2=\sum_{j=1}^{n}(|s_j|^2+|t_j|^2)|x_j|^2$.
In particular, we have
\begin{equation}\label{equ norm-of-T}
  \|T\|^2=\max_{1\leqslant j\leqslant n}|s_j|^2+|t_j|^2.
\end{equation}
Let $A=aL_{v_j}+bL_{e_j}$.
Then $\delta(A)=(bs_j-at_j)L_{e_j}$.
It follows that
\begin{equation*}
  |bs_j-at_j|^2=\|\delta(A)\|^2\leqslant\|\delta\|^2\|A\|^2
  =\|\delta\|^2(|a|^2+|b|^2).
\end{equation*}
Let $a=-\overline{t_j}$ and $b=\overline{s_j}$.
Then $|s_j|^2+|t_j|^2\leqslant\|\delta\|^2$.
Hence $\|T\|\leqslant\|\delta\|$ by \eqref{equ norm-of-T}.
This completes the proof.
\end{proof}

At the end of this subsection, we propose the following conjecture.

\begin{conjecture}\label{conj in-tree}
Suppose $G$ is an in-tree.
Then for every derivation $\delta$ on $\mathcal{A}_G$, there exists $T\in\mathcal{A}_G$ such that $\delta=\delta_T$ and $\|T\|\leqslant C\|\delta\|$, where $C\geqslant 1$ is a universal constant.
\end{conjecture}

\subsection{Generalized trees}
For any tree $G$, we have showed in \Cref{lem tree} that for every derivation $\delta$ on $\mathcal{A}_G$, there exists $T\in\mathcal{A}_G$ such that $\delta=\delta_T$.
Under the assumption of \Cref{conj in-tree}, we are able to give an upper bound of the norm of $T$ with respect to the alternating number $A(G)$.
Recall that $\mathcal{A}_G$ is finite-dimensional for every tree $G$.

\begin{proposition}\label{prop tree-bound}
Suppose $G$ is a tree with alternating number $A(G)=n$.
Assume that \Cref{conj in-tree} holds.
Then for every derivation $\delta$ on $\mathcal{A}_G$, there exists $T\in\mathcal{A}_G$ such that $\delta=\delta_T$ and $\|T\|\leqslant C_n\|\delta\|$, where $C_n\geqslant 1$ is a universal constant (only depends on $n$).
\end{proposition}

\begin{proof}
We prove the lemma by induction on $n$.
If $n=0$, then $G$ is an upper triangle tree.
Hence we can take $C_0=1$ by \Cref{prop out-tree}.

Assume that the lemma holds for some integer $n\geqslant 0$.
Let $G$ be a tree with $A(G)=n+1$ and $G_0$ a maximal subtree of $G$ with $A(G_0)=n$.
Suppose the reduced subgraph of $G$ by $\mathcal{V}\backslash\mathcal{V}_0$ has connected components $\{G_j\}_{j=1}^m$.
Let $\mathcal{V}_j=\mathcal{V}(G_j)$ for $0\leqslant j\leqslant m$.
Then there is a unique edge $e_j$ between $v_j\in\mathcal{V}_j$ and $x_j\in\mathcal{V}_0$ for every $1\leqslant j\leqslant m$.
By the maximality of $G_0$, each $G_j$ is a rooted tree with root $v_j$.
Otherwise, $A(G)\geq n+2$, which is a contradiction.
Without loss of generality, we assume that $G_j$ is an in-tree for each $1\leqslant j\leqslant m'$ and it is an out-tree for each $m'+1\leqslant j\leqslant m$.
It follows that $s(e_j)=v_j$ for each $1\leqslant j\leqslant m'$ and $r(e_j)=v_j$ for each $m'+1\leqslant j\leqslant m$.
We define $T_0$ and $\delta_1$ as in the proof of \Cref{prop reduction}, where $\|T_0\|\leqslant 2\|\delta\|$.
By assumption, there exists $S_0\in Q_0\mathcal{A}_GQ_0$ such that $\delta_1|_{Q_0\mathcal{A}_GQ_0}=\delta_{S_0}$ and $\|S_0\|\leqslant C_n\|\delta_1\|$.
By \Cref{prop out-tree} and \Cref{conj in-tree}, we can similarly define $S$ and $\delta_2$, where $\|S\|\leqslant\max\{C,C_n\}\|\delta_1\|$.
At last, we have $\delta_2(L_{e_j})=\lambda_jL_{e_j}$ for every $1\leqslant j\leqslant m$.
Therefore, there exists $B\in\mathcal{A}_G$ such that $\delta_2=\delta_B$ and $\|B\|\leqslant\|\delta_2\|$.
This completes the proof.
\end{proof}

Note that $A(G_0)\leqslant A(G)$ for every subgraph $G_0$ of $G$.
Similar to \Cref{cor reduction}, we can get the following consequence.

\begin{corollary}\label{cor tree-bound}
Suppose $G$ is a generalized tree with alternating number $A(G)=n$.
Assume that \Cref{conj in-tree} holds.
Then for every bounded derivation $\delta$ from $\mathcal{A}_G$ into $\mathfrak{L}_G$, there exists $T\in\mathfrak{L}_G$ such that $\delta=\delta_T$ and $\|T\|\leqslant C_n\|\delta\|$, where $C_n\geqslant 1$ is a universal constant (only depends on $n$).
\end{corollary}

Assume that \Cref{conj in-tree} holds.
By \Cref{prop derivation-decomposition}, \Cref{thm S-C}, \Cref{cor reduction}, and \Cref{cor tree-bound}, we obtain a converse of \Cref{cor inner-imply}.

\begin{corollary}\label{cor fruit-tree-bound}
Suppose $G$ is a directed graph with connected components $\{G_{\lambda_1}\}_{\lambda_1\in\Lambda_1}
\cup\{G_{\lambda_2}\}_{\lambda_2\in\Lambda_2}$ such that each $G_{\lambda_1}$ is strongly connected and each $G_{\lambda_2}$ is not strongly connected.
Assume that \Cref{conj in-tree} holds.
If each $G_{\lambda_2}$ is a generalized fruit tree and $n:=\sup_{\lambda_2\in\Lambda_2}A(G_{\lambda_2})<\infty$, then every bounded derivation from $\mathcal{A}_G$ into $\mathfrak{L}_G$ is inner.

More precisely, for every bounded derivation $\delta$ from $\mathcal{A}_G$ into $\mathfrak{L}_G$, there exists $T\in\mathfrak{L}_G$ such that $\delta=\delta_T$ and $\|T\|\leqslant C_n'\|\delta\|$, where $C_n'\geqslant 1$ is a universal constant (only depends on $n$).
\end{corollary}

\section{The first cohomology groups}\label{sec cohomology}
We have shown that for a finite connected directed graph $G$, the first cohomology group $H^1(\mathcal{A}_G,\mathfrak{L}_G)$ vanishes if and only if $G$ is either strongly connected or a fruit tree by \Cref{prop S-C-finite}, \Cref{thm fruit-tree}, and \Cref{thm converse}.
In this section, we present some examples such that the first cohomology groups are nontrivial.

By \Cref{def fruit-tree}, the directed graph in the following example is not a fruit tree.
It follows that the first cohomology group $H^1(\mathcal{A}_G,\mathfrak{L}_G)$ is not zero by \Cref{thm converse}.

\begin{example}\label{eg 2v-2l-1e}
Suppose $G$ is the directed graph with two vertices $\{v_1,v_2\}$, one edge $e$ such that $r(e)=v_1$ and $s(e)=v_2$, and two loops $e_1$ and $e_2$ such that $r(e_1)=s(e_1)=v_1$ and $r(e_2)=s(e_2)=v_2$.
See the following graph
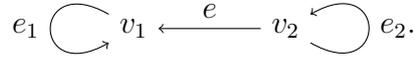
\begin{center}
\begin{tikzpicture}
    \node (x) at (0,0) {$v_1$};
    \node (y) at (2,0) {$v_2$};
    \draw[->] (y) -- node[above] {$e$} (x);
    \draw[->] (x) to [out=150,in=210,loop] node[left] {$e_1$} (x);
    \draw[->] (y) to [out=-30,in=30,loop] node[right] {$e_2$.} (y);
\end{tikzpicture}
\end{center}
Then
\begin{equation*}
  H^1(\mathcal{A}_G,\mathfrak{L}_G)
  =\frac{H^\infty(\mathbb{D}^2)}{H^\infty(\mathbb{D})\otimes 1+1\otimes H^\infty(\mathbb{D})}.
\end{equation*}
\end{example}

\begin{proof}
Let $\mathscr{A}=\{L_{v_1},L_{v_2}\}$.
Then every derivation from $\mathcal{A}_G$ into $\mathfrak{L}_G$ is locally inner with respect to $\mathscr{A}$ by \Cref{lem idempotent}.
Let $\delta$ be a bounded derivation from $\mathcal{A}_G$ into $\mathfrak{L}_G$ such that $\delta|_{\mathscr{A}}=0$.
Then $\delta$ is a derivation from $L_{v_j}\mathcal{A}_GL_{v_j}$ into $L_{v_j}\mathfrak{L}_GL_{v_j}$ for $j=1,2$.
Since $L_{v_j}\mathcal{A}_GL_{v_j}\cong A(\mathbb{D})$ and $L_{v_j}\mathfrak{L}_GL_{v_j}\cong H^\infty(\mathbb{D})$ by \Cref{prop reduced}, we have $\delta|_{L_{v_j}\mathcal{A}_GL_{v_j}}=0$ for $j=1,2$.
In particular, $\delta(L_{e_j})=0$ for $j=1,2$.
Therefore, $\delta$ is uniquely determined by $\delta(L_e)$.

For any operator $A\in L_{v_1}\mathfrak{L}_GL_{v_2}$, we can write
\begin{equation}\label{equ A-fn}
  A=\sum_{n=0}^{\infty}L_{e_1^ne}f_n(L_{e_2}),\quad
  f(z,w):=\sum_{n=0}^{\infty}f_n(z)w^n,
\end{equation}
where $f\in H^2(\mathbb{D}^2)$.
For any function $\varphi\in H^2(\mathbb{D})$, we have
\begin{equation*}
  \|A\varphi(L_{e_2})\xi_{v_2}\|^2
  =\sum_{n=0}^{\infty}\|f_n\varphi\|^2_{H^2(\mathbb{D})}
  =\int_{\mathbb{T}}\sum_{n=0}^{\infty}|f_n(z)|^2|\varphi(z)|^2d\mu(z),
\end{equation*}
where $\mu$ is the Haar measure on $\mathbb{T}$.
It follows that
\begin{equation*}
  \|A\|=\sup_{z\in\mathbb{D}}\sum_{n=0}^{\infty}|f_n(z)|^2
  =\sup_{z\in\mathbb{D}}\|f(z,w)\|_{H^2(\mathbb{D}),w}^2.
\end{equation*}
Since $\delta(L_{v_j})=0$ for $j=1,2$, we can write
\begin{equation*}
  \delta(L_e)=\sum_{n=0}^{\infty}L_{e_1^ne}g_n(L_{e_2}),\quad
  g(z,w):=\sum_{n=0}^{\infty}g_n(z)w^n.
\end{equation*}
If $A$ is of the form \eqref{equ A-fn}, then
\begin{equation*}
  \delta(A)=\sum_{n=0}^{\infty}L_{e_1^ne}
  \sum_{m=0}^{n}f_m(L_{e_2})g_{n-m}(L_{e_2}).
\end{equation*}
It follows that $\|\delta(A)\|=\sup_{z\in\mathbb{D}}\|g(z,w)f(z,w)\|_{H^2(\mathbb{D}),w}^2$.
Therefore, we have
\begin{equation*}
  \|\delta\|=\|g(z,w)\|_{H^\infty(\mathbb{D}^2)}.
\end{equation*}
Let $T\in\mathfrak{L}_G$.
If $\delta_T(L_{v_j})=0$ for $j=1,2$, then $T=f(L_{e_1})+g(L_{e_2})$ and
\begin{equation*}
  \delta_T(L_e)=L_eg(L_{e_2})-f(L_{e_1})L_e.
\end{equation*}
It follows that
\begin{equation*}
  \|\delta_T\|=\|g(z)-f(w)\|_{H^\infty(\mathbb{D}^2)}
  =\|f\|_{H^\infty(\mathbb{D})}+\|g\|_{H^\infty(\mathbb{D})}.
\end{equation*}
By \Cref{prop zero-at-A}, we obtain that
\begin{align*}
    H^1(\mathcal{A}_G,\mathfrak{L}_G)
    &\cong \mathrm{Der}_{\mathscr{A}}(\mathcal{A}_G,\mathfrak{L}_G)
    /\mathrm{Inn}_{\mathscr{A}}(\mathcal{A}_G,\mathfrak{L}_G)\\
    &\cong H^{\infty}(\mathbb{D}^2)/\{f(z)+g(w)\colon f,g\in H^{\infty}(\mathbb{D})\}.
\end{align*}
The proof is completed.
\end{proof}

\begin{example}\label{eg 2v-ne}
Suppose $G$ is a directed graph with two vertices $\{v_1,v_2\}$, and $n$ edges $\{f_j\}_{j=1}^n$ from $v_2$ to $v_1$.
Then
\begin{equation*}
  H^1(\mathcal{A}_G,\mathfrak{L}_G)\cong M_n(\mathbb{C})/\mathbb{C}I_n.
\end{equation*}
\end{example}

\begin{proof}
Let $\mathscr{A}=\{L_{v_1},L_{v_2}\}$.
Similar to \Cref{eg 2v-2l-1e}, we assume that $\delta|_{\mathscr{A}}=0$.
Then $\delta$ is uniquely determined by $\{\delta(L_{f_j})\}_{j=1}^n$.
Assume that
\begin{equation*}
  \delta(L_{f_i})=\sum_{j=1}^{n}a_{ij}L_{f_j}.
\end{equation*}
Then $\delta$ is determined by $(a_{ij})\in M_n(\mathbb{C})$.
If $T\in\mathfrak{L}_G$ and $\delta_T(L_{v_j})=0$ for $j=1,2$, then $T=\lambda_1L_{v_1}+\lambda_2L_{v_2}$.
It follows that
\begin{equation*}
  \delta_T(L_{f_i})=(\lambda_2-\lambda_1)L_{f_i}.
\end{equation*}
Hence $\delta_T$ is corresponding to the matrix $(\lambda_2-\lambda_1)I_n\in\mathbb{C}I_n$.
The conclusion follows from \Cref{prop zero-at-A}.
\end{proof}

\begin{example}\label{eg 3v-1l-2e}
Suppose $G$ is the directed graph with three vertices $\{v_1,v_2,v_3\}$, two edges $e_1$ and $e_2$ such that $r(e_1)=r(e_2)=v_3$, $s(e_1)=v_1$, and $s(e_2)=v_2$, and one loop $w$ such that $r(w)=s(w)=v_3$.
See the following graph
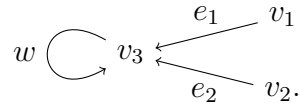
\begin{center}
\begin{tikzpicture}
    \node (x) at (0,0) {$v_3$};
    \node (y) at (2,0.5) {$v_1$};
     \node (z) at (2,-0.5) {$v_2.$};
    \draw[->] (y) -- node[above] {$e_1$} (x);
     \draw[->] (z) -- node[below] {$e_2$} (x);
    \draw[->] (x) to [out=150,in=210,loop] node[left] {$w$} (x);
\end{tikzpicture}
\end{center}
Then $ H^1(\mathcal{A}_G,\mathfrak{L}_G)=H^{\infty}(\mathbb{D})/\mathbb{C}$.
\end{example}

\begin{proof}
Let $\mathscr{A}_1=\{L_{v_1},L_{v_2},L_{v_3}\}$.
Similar to \Cref{eg 2v-2l-1e}, we may assume that $\delta|_{\mathscr{A}_1}=0$.
In this case, $\delta$ is a derivation from $L_{v_3}\mathcal{A}_GL_{v_3}$ into $L_{v_3}\mathfrak{L}_GL_{v_3}$.
Since $L_{v_3}\mathfrak{L}_GL_{v_3}\cong H^\infty(\mathbb{D})$ is commutative, we have $\delta|_{L_{v_3}\mathfrak{L}_GL_{v_3}}=0$ by \Cref{thm 4.5}.
In particular, $\delta(L_w)=0$.
Moreover, $\delta(L_{e_1})=-f(L_w)L_{e_1}$ for some $f\in H^{\infty}(\mathbb{D})$ by the proof of \Cref{thm fruit-tree}.
Then $\delta$ and $\delta_{f(L_w)}$ coincide on $\mathscr{A}=\mathscr{A}_1\cup\{L_w,L_{e_1}\}$.
Hence $\delta$ is locally inner with respect to $\mathscr{A}$.

We may assume that $\delta|_{\mathscr{A}}=0$.
Then $\delta$ is uniquely determined by
\begin{align*}
    \delta(L_{e_2})=-g(L_w)L_{e_2},\quad g\in H^{\infty}(\mathbb{D}).
\end{align*}
Then $\delta$ is determined by the function $g\in H^{\infty}(\mathbb{D})$.
If $T\in\mathfrak{L}_G$ and $\delta_T|_{\mathscr{A}}=0$,
then $T=\alpha(L_{v_1}+L_{v_3})+\beta L_{v_2}$, where $\alpha,\beta\in\mathbb{C}$.
It follows that
\begin{equation*}
  \delta_T(L_{e_2})=(\beta-\alpha)L_{e_2}.
\end{equation*}
Hence $\delta_T$ is corresponding to the scalar $\beta-\alpha\in\mathbb{C}$ .
The conclusion follows from \Cref{prop zero-at-A}.
\end{proof}

\end{document}